\documentclass[11pt]{amsart}

\pdfoutput=1
\usepackage{preambolo}

\title{Rational circle-equivariant elliptic cohomology of $\cp V$}
\author{Matteo Barucco}
\address{Department of Mathematics, University of Warwick}
\email{Matteo.Barucco@Warwick.ac.uk}
\begin{document}
	\begin{abstract}
		We compute rational $\T$-equivariant elliptic cohomology of $\cp V$, where $\T$ is the circle group, and $\cp V$ is the $\T$-space of complex lines for a finite dimensional complex $\T$-representation $V$. Starting from an elliptic curve $\c$ over $\C$ and a coordinate data around the identity, we achieve this computation by proving that the $\T$-equivariant elliptic cohomology theory $E\c_{\T}$ built in \cite{john:elliptic}, and the $\T^2$-equivariant elliptic cohomology theory $E\c_{\T^2}$ built in \cite{barucco:t2} are $H_1$-split, for the subgroup $H_1=1\times \T<\T^2$. This result allows us to reduce the computation of $\ect(\cp V)$ to the computation of $\T^2$-elliptic cohomology of spheres of complex representations, already performed in \cite[Theorem 1.4]{barucco:t2}.
	\end{abstract}
\maketitle
\setcounter{tocdepth}{1}
\tableofcontents
\section{Introduction}
\subsection{Motivation}
	Two of the most prominent examples of topological cohomology theories have a one dimensional formal group associated: ordinary cohomology is associated to the additive formal group, while complex $K$-theory is associated to the multiplicative one. More generally one can associate to every complex orientable cohomology theory a one dimensional formal group \cite{quillen:formal}. A great source of formal groups is the formal completion of a one dimensional algebraic group at the identity. If we are working over an algebraically closed field for example, we are soon in short supply of one dimensional algebraic groups, except for the additive and multiplicative ones all the others are elliptic curves. These more exotic theories whose associated formal group arise as the formal completion of an elliptic curve $\c$ around the identity are called elliptic cohomology theories, and were first constructed in \cite{lan:periodic} \cite{lan:ell}.

	Given a compact Lie group $G$, it is only natural to seek for a $G$-equivariant counterpart of elliptic cohomology. If one takes complex K-theory as a model, than we can expect a good theory of equivariant elliptic cohomology to encode the full algebraic group $\c$, in contrast with the non-equivariant theory simply encoding the formal completion of $\c$ around the identity. This is in analogy to how equivariant K-theory works. By the Atiyah-Segal completion theorem \cite{atiyah:complet} if we complete the equivariant K-theory of the point at its augmentation ideal, we obtain the K-theory of $BG$, the classifying space of $G$:
	\[
	KU_G(*)^\wedge_I\cong KU(BG).
	\]
	Equivariant elliptic cohomology was first constructed by Grojnowski in \cite{groj:delocalised} with complex coefficients, and the specific motivation to study certain elliptic algebras. The subject has thrived since then and has found numerous applications to representation theory (for $G$ connected compact Lie group \cite{groj:delocalised}, \cite{ginzburg:elliptic}), to organize moonshine phenomena (especially for $G$ finite \cite{dev:delocalised}, \cite{dev:algebraicdescr}), and also quite surprisingly to physics \cite{stolz}, \cite{daniel:supersymmetric}.
	
	Given the complexity of defining Equivariant elliptic cohomology (\cite{lurie:ellipticI}, \cite{lurie:ellipticII}, \cite{lurie:ellipticIII}, \cite{gepner:onequiv}), and its algebraic origin also in the non-equivariant case, it is quite difficult to compute $G$-equivariant elliptic cohomology of meaningful $G$-spaces. Computations are often limited to fixed $G$-spaces or spheres of complex representations \cite{john:elliptic}, \cite{barucco:t2} to check that the theory defined is indeed elliptic cohomology. This paper has the goal to enhance computations at least for rational $\T$-equivariant elliptic cohomology, and in particular given an elliptic curve $\c$ and a finite dimensional complex $\T$-representation $V$ to compute rational $\T$-equivariant elliptic cohomology of $\cp V$: the $\T$-space of complex lines in $V$.
	
	More precisely we will prove the following which is the main theorem of the paper. 
	\begin{theorem}
		\label{thm:iso2}
		For every elliptic curve $\c$ over $\C$, if $\ect$ \eqref{eq:ecgbar} is the rational $\T$-equivariant elliptic cohomology theory built in \cite{john:elliptic}, and $V$ is a finite dimensional complex representation of $\T$, then:
		\begin{enumerate}
			\item If $V$ has one isotypic component, $V=\alpha z^n$ with $\alpha \geq 1$,
			\begin{equation*}
				\ect^k(\cp V)\cong \C^{\alpha-1}
			\end{equation*}
			for every $k\in \Z$.
			\item If $V$ has more than one isotypic component, $V=\bigoplus_n\alpha_n z^n$,
			\begin{equation*}
				\ect^k(\cp V)\cong 
				\begin{cases*}
					0 & $k$ \text{even} \\
					\C^d & $k$ \text{odd}.
				\end{cases*}
			\end{equation*}
			where $d=\sum_{i<j}\alpha_i\alpha_j(i-j)^2$, and $z$ is the natural representation of $\T$.
		\end{enumerate}
	\end{theorem}
	Part (1) of this theorem simply checks our methods on previously done computations since $\cp {\alpha z^n}\cong \Cp^{\alpha-1}$ with the trivial $\T$-action, and can be found in the paper as Lemma \ref{thm:iso}. In contrast, part (2) is the main computational result, can be found in the paper as Theorem \ref{thm:mainthm} and to the knowledge of the author it is the first time it appears in the literature. We refer to Remark \ref{rem:dictionary} for a discussion of the geometry of this Theorem. We revise the construction of $\ect$ in Section \ref{sec:cpvcirclecase}.
	
	Notice that there are different representations, with different complex projective spaces that nonetheless have the same value for $d$ and therefore rational $\T$-equivariant elliptic cohomology does not distinguish them. For example if $\epsilon$ is $\C$ with the trivial action, $V=\epsilon \oplus 4z$ and $V'=\epsilon \oplus z^2$ have $d=4$, or $W=\epsilon\oplus 16z$, $W'=\epsilon\oplus z^4$ and $W''=\epsilon \oplus z\oplus 3z^2$ have $d=16$. Therefore they have the same elliptic cohomology even if their complex projective spaces are quite different.
	
	In Theorem \ref{thm:iso2} we use the construction of rational $\T$-equivariant elliptic cohomology of Greenlees \cite[Theorem 1.1]{john:elliptic} with algebraic models. Algebraic models have the advantage to simplify the picture and to maintain a more evident bond with the geometry of the curve $\c$, which are essential aspects when performing computations. This construction of rational equivariant elliptic cohomology using algebraic models has been generalized for the $2$-torus $\T^2$ \cite{barucco:t2}, and we will make essential use of the $\T^2$-theory as well in the computation.
    
\subsection{Algebraic Models}
	Algebraic models are a useful tool to study equivariant cohomology theories: equivariant $K$-theory, equivariant cobordism, Bredon and Borel cohomology just to name a few examples. For every compact Lie group $G$ one can construct a category of $G$-spectra where every cohomology theory $E_G^*(\_)$ is represented by a $G$-spectrum $E$, in the sense that for any based $G$-space $X$ we have $E_G^*(X)=[\Sigma^{\infty}X, E]^G_*$. To simplify the picture one can restrict the focus to rational $G$-equivariant cohomology theories, represented by rational $G$-spectra: the localization of $G$-spectra at the rational sphere spectrum. 
    
    The main idea of algebraic models \cite{john:mem} is for $G$ a compact Lie group to define an abelian category $\A(G)$ and an homology functor
    \begin{equation}
    	\label{eq:homologyfunctor}
    	\pi_*^{\A}:\sp G_{\Q} \to \A(G)
    \end{equation}
    together with a strongly convergent Adams spectral sequence to compute stable maps in rational $G$-spectra:
    \begin{equation}
    	\label{eq:ass}
    	\Ext^{*,*}_{\A}(\pi_*^\A(X),\pi_*^\A(Y)) \Longrightarrow [X,Y]^G_{*}
    \end{equation}
    Furthermore for certain groups the homology functor $\pi_*^{\A}$ can be lifted to a Quillen-equivalence 
    \begin{equation}
    	\label{eq:Quillen}
    	\sp G_\Q \simeq_Q d\A(G)
    \end{equation}
    between rational $G$-spectra and the category $d\A(G)$ of differential graded objects in $\A(G)$. The Adams spectral sequence \eqref{eq:ass} is a powerful tool to compute values of known theories, while the Quillen-equivalence \eqref{eq:Quillen} can be used to build entirely new theories, simply constructing  objects in $d\A(G)$. This latest method is precisely the one used in \cite{john:elliptic} and \cite{barucco:t2}: building an object $E\c_{G}$ in $d\A(G)$ and using the Quillen-equivalence \eqref{eq:Quillen} for the circle $G=\T$ and the $2$-torus $G=\T^2$ to define a $G$-equivariant elliptic cohomology theory. This method works since Greenlees and Shipley have extended the Quillen-equivalence \eqref{eq:Quillen} to tori of any rank $\T^r$ \cite[Theorem 1.1]{john:torusship}. 
    
    \subsection{The method of this paper}
    We start from an elliptic curve $\c$ over $\C$ together with a coordinate function $t_e$ in the local ring at the identity $e$ of $\c$. This gives us a rational $\T$-equivariant elliptic cohomology theory $\ect\in \A(\T)$ \cite[Theorem 1.1]{john:elliptic}, as well as a rational $\T^2$-equivariant elliptic cohomology theory $\ectwo\in \A(\T^2)$ \cite[Theorem 1.4]{barucco:t2}. We have also the complex abelian surface $\X=\c\times \c$.
    
    Given a complex $\T$-representation $V$ the main goal of this paper is to compute the reduced cohomology of the pointed space:
    \[
    \ect^*(\cp V).
    \]
 	We start pointing out the isomorphism of $\T$-spaces:
    \begin{equation}
    	\label{eq:cpv3}
    	\cp V\cong S(V\otimes_{\C} w)/H_1
    \end{equation}
    where $w$ is the natural complex representation of $H_1=1\times \T<\T^2$ and $S(V\otimes_{\C} w)$ is the $\T$-space of vectors of unit norm in the complex vector space $V\otimes_{\C} w$. Notice that $V\otimes_{\C} w$ is a complex representation of $\T^2$ of the same dimension of $V$, where $H_1$ acts on the second factor of the tensor product. The computation is made possible since $\ect$ and $\ectwo$ are $H_1$-split, where we view $\T$ as the quotient group $\T^2/H_1$:
    
    \begin{theorem}
    	\label{thm:split2}
    	There is a natural transformation of $\T^2$-cohomology theories
    	\begin{equation}
    		\label{eq:epsilonfirst}
    		\epsilon: \inf {\T}{\T^2} \ect \longrightarrow \ectwo
    	\end{equation}
    	which induces an isomorphism
    	\begin{equation}
    		\label{eq:equivalence2}
    		[\quotient {\T^2}H_+, \inf {\T}{\T^2} \ect]^{\T^2}_*\cong [\quotient {\T^2}H_+, \ectwo]^{\T^2}_*
    	\end{equation}
    	for every subgroup $H$ of $\T^2$ such that $H\cap H_1=\{1\}$.
    \end{theorem}
    We prove this Theorem in Sections \ref{sec:cpvmap} and \ref{sec:cpvsplit}. More precisely in Section \ref{sec:cpvmap} we build the map $\epsilon$ while in Section \ref{sec:cpvsplit} we prove the $H$-equivalence \eqref{eq:equivalence2}. 
    As an immediate Corollary:
    \begin{corollary}
    	\label{cor:h1quot2}
    	For any $H_1$-free $\T^2$-space $X$:
    	\begin{equation}
    		\ectwo^*(X)\cong \ect^*(X/H_1)
    	\end{equation}
    \end{corollary}
    
    We can apply this corollary to  \eqref{eq:cpv3}, reducing the computation of $\T$-equivariant elliptic cohomology to a computation of $\T^2$-equivariant elliptic cohomology:
    \begin{equation}
    	\label{eq:gequiv2}
    	\ect^* (\cp V_+)\cong \ect^*(\quotient {S(V\otimes w)_+}{H_1})\cong \ectwo^*(S(V\otimes w)_+).
    \end{equation}
    
    The $\T^2$-equivariant elliptic cohomology of $S(V\otimes w)_+$ is easier to compute in view of the cofibre sequence of $\T^2$-spaces:
    \begin{equation*}
    	S(V\otimes w)_+ \longrightarrow S^0 \longrightarrow S^{V\otimes w}
    \end{equation*}
    inducing a long exact sequence:
    \begin{equation}
    	\label{eq:leselliptic2}
    	\ectwo^*(S^{V\otimes w}) \longrightarrow \ectwo^*(S^0) \longrightarrow  \ectwo^*(S(V\otimes w)_+).
    \end{equation}
    The first two terms of \eqref{eq:leselliptic2} are computed in \cite[Theorem 1.4]{barucco:t2} that we recall here for easy reference:
    \begin{theorem}
    	\label{thm:fundconstr2}
    	For every elliptic curve $\c$ over $\C$ and coordinate $t_e\in \O_{\c,e}$, there exists an object $E\c_{\T^2}\in \A(\T^2)$ whose associated rational $\T^2$-equivariant cohomology theory $E\c_{\T^2}^*(\_)$ is 2-periodic. The value on the one point compactification $S^W$ for a complex $\T^2$-representation $W$ with no fixed points is given in terms of the sheaf cohomology of a line bundle $\odw$ over $\X=\c\times \c$:
    	\begin{equation}
    		\label{eq:valuetheory}
    		\ectwo^n(S^W)\cong 
    		\begin{cases*}
    			H^0(\X, \odw)\oplus H^2(\X, \odw) & $n$ even \\
    			H^1(\X, \odw) & $n$ odd.
    		\end{cases*}
    	\end{equation}
    \end{theorem}
The associated divisor $D_W$ is defined as follows. First our elliptic curve $\c$ defines a functor $\XF$ from compact abelian Lie groups to complex manifolds. If $H$ is a compact abelian Lie group and $\c$ our fixed elliptic curve, define
\[
\XF(H):=\Hom_{\text{Ab}}(H^*,\c).
\]
Where we are considering group homomorphisms, and $H^*:= \Hom_{\text{Lie}}(H, \T)$ is the character group of $H$. Note this is an exact functor inducing an embedding $\XF(K)\hookrightarrow \XF(H)$ for every embedding $K\hookrightarrow H$.
The fundamental algebraic surface we are working on is $\X:=\XF(\T^2)\cong \c\times \c$. Note $\XF(H)$ is a subvariety of $\X$ of the same dimension of $H$. For every non-trivial character $z'$ of $\T^2$ consider the codimension $1$ subgroup $\Ker(z')$, then $\XF(\Ker(z'))$ is a codimension $1$ subvariety of $\X$ that is the associated divisor of $z'$. If $W=\sum_{z'}\alpha_{z'}{z'}$ define
\begin{equation}
	\label{eq:assdiv}
	D_W:= \sum_{z'} \alpha_{z'}\XF(\Ker(z')),
\end{equation}
where in both sums $z'$ runs over all the non-trivial characters of $\T^2$, and $\alpha_{z'}$ is a non-negative integer.
\begin{rem}
	One might argue that to compute rational $\T$-equivariant elliptic cohomology of $\cp V$ a more straightforward approach can be taken. Namely we have $\ect \in \A(\T)$ and if we explicitly compute the algebraic model $\pi_*^{\A(\T)}(\cp V_+)$, then we can use the Adams spectral sequence for the circle to compute directly $\ect^*(\cp V_+)$. Originally this was the method tried on this project, with a double motivation: first as a warm-up in the circle case before generalizing to higher tori, and second, once $\ectwo$ has been constructed, compute $\ectwo^*(S(V\otimes w)_+)$, and see if the two values match. If indeed $\ect^*(\cp V_+)\cong \ectwo^*(S(V\otimes w)_+)$, then one might expect that $\ect$ and $\ectwo$ are $H_1$-split (Theorem \ref{thm:split2}). The problem encountered with this method is that after computing the algebraic model $\pi_*^{\A(\T)}(\cp V_+)$ then computing maps in $\A(\T^2)$ from this object into an injective resolution of $\ect$ still retains a lot of complexity. More precisely it is still difficult to compute $\of$-module maps between arbitrary $\of$-module when none of them is a suspension of the base ring $\of$. Therefore this more conceptual approach of proving $H_1$-splitness and computing $\ectwo^*(S(V\otimes w)_+)$ has been taken. Moreover we hope to replicate and generalise this to other spaces such as Grassmannians of $n$-planes $\Gr_n(V)$ of a complex $\T$-representation.
\end{rem}

\subsection{Outline of the paper} 
The paper is organized as follows. In the interest of clarity we start in Section \ref{sec:1} with the main computational result: namely we compute rational $\T$-equivariant elliptic cohomology of $\cp V$ (Lemma \ref{thm:iso} and Theorem \ref{thm:mainthm}) assuming $H_1$-splitness (Theorem \ref{thm:split}), and we devote the remaining of the paper in proving $H_1$-splitness. 
In Section \ref{sec:cpvcirclecase} we revise the construction of circle equivariant elliptic cohomology from \cite{john:elliptic} to make it fit nicely with the construction of the $\T^2$-equivariant theory from \cite{barucco:t2}.
In Section \ref{sec:cpvmap} we construct the map $\epsilon$ (Lemma  \ref{lem:epsilon}) between the inflated $\ect$ and $\ectwo$, inducing the $H_1$-splitness. To do so we will also treat the inflation functor in the algebraic models for general tori, proving a simple construction of the functor in the algebraic models (Proposition \ref{prop:inflatedsystem}).
In Section \ref{sec:cpvsplit} we conclude the proof of $H_1$-splitness by proving the isomorphism \eqref{eq:equivalence2} for every subgroup $H$ intersecting $H_1$ trivially. This is done in Theorem \ref{thm:extgroupiso} by proving that $\epsilon$ induces an isomorphisms between the $\Ext$-groups of the two Adams spectral sequences. To obtain this result will be fundamental to compute the algebraic model for natural cells (Lemma \ref{lem:naturalcells} and Lemma \ref{lem:gf}), and to build an injective resolution of the inflated $\ect$ \eqref{eq:infres}.
We end with an Appendix on algebraic models recalling from the literature the definition of the category $\A(G)$, and especially injective objects in it.

\subsection{Notation and Conventions} 
The most important convention of the paper is that everything is rationalized without comment. In particular all spectra are meant localized at the rational sphere spectrum $S^0\Q$ and all the homology and cohomology is meant with $\Q$ coefficients. Tensor products $\otimes$ are meant over $\Q$, or over the graded ring with only $\Q$ in degree zero and zero elsewhere. By subgroup of a compact Lie group we always mean closed subgroup. With the symbol $\T^r$ we mean the torus of rank $r$: compact connected abelian Lie group of dimension $r$. We denote $G=\T^2$ the $2$-torus in all the paper (except where explicitly state otherwise like in the Appendix). The collection of connected closed codimension $1$ subgroups of $G$ is $\{H_i\}_{i\geq 1}$ indexed with $i\geq 1$, and with $H_1=1\times \T$ and $H_2=\T\times 1$ being the two privileged subgroups. We denote $\bar G:=G/H_1\cong H_2\cong \T$ the quotient group. We denote $H_i^j$ the subgroup of $G$ with $j$ connected components and identity component $H_i$: we will refer to the subgroups with identity component $H_i$ as being along the $i$-direction. We denote with $F$ a generic finite subgroup of $G$ and with $H$ and $K$ generic closed subgroups of $G$ of any dimension. Given a module $M$ we will denote $\overline M$ the two periodic version of $M$: it is a graded module with $M$ in each even degree and zero in odd degrees. We denote elements in direct sums and products in the following way: $x=\{x_i\}_{i}\in \bigoplus_{i\geq 1}M_i$: this identifies the element $x$ in the direct sum that has $i$-th component $x_i\in M_i$.

We will freely use the standard notation of schemes from Algebraic Geometry. We denote $\c$ our fixed elliptic curve over $\C$, $e$ is the identity of the elliptic curve and for a positive integer $n$, $\c[n]$ is the subgroup of elements of $n$-torsion, while $\c\gen n$ is the subset of elements of exact order $n$. We will use $P$ to denote a point of $\c$ of finite order. We have the complex abelian surface $\X=\c\times \c$, and denote $\eta(C)$ the generic point of a closed set $C$.

We will use some essential notation involving $\ectwo$ from \cite{barucco:t2}. Following \cite[(2.7), (2.8)]{barucco:t2}, for every $i\geq 1$ we pick a character $z_i:G \to \T$ having $H_i$ as kernel. If $z_i(x,y)=x^{\lambda_i}y^{\mu_i}$ we denote $\pi_i:\X \to \c$ the corresponding projection: $\pi_i(P,Q)=\lambda_iP+\mu_iQ$. Then for every $i,j\geq 1$ we can define the codimension 1 subvariety $\dij:=\pi_i^{-1}(\c\gen j)$ of $\X$. The most important ring for $\ectwo$ is $\K$: the ring of meromorphic functions of $\X$ with poles only in the collection $\{\dij\}_{ij}$ \cite[(3.8)]{barucco:t2}. We use $\O_{\dij}$ to denote the subring of $\K$ of those functions regular at the closed set $\dij$, while  $m_{ij}<\O_{\dij}$ is the ideal of those functions vanishing at $\dij$. 

We will freely use the standard notation for algebraic models , and we recall it in the appendix \ref{appendix:algmodel}. In particular $\A(G)$ is an abelian category with graded objects and no differentials, while $d\A(G)$ is the category of objects of $\A(G)$ with differentials. Cohomology is unreduced unless indicated to the contrary with a tilde, so that $\hbg H=\tilde H^*(BG/H_+)$ is the unreduced cohomology ring. To ease the notation sometimes we will omit the base ring we are taking the tensor product over and denote it with an index: $\tens i$. This in turn means the ring $\oefh i$ or its $F$-component $\hbg {H_i^{n_i}}$. We will make extensive use of the isomorphisms $\hbg {H_i^j}\cong \Q[c_{ij}]$ \eqref{eq:cij} and $\hbg F\cong \Q[\xa,\xb]$ \eqref{eq:xaxb}.

\subsection{Acknowledgements} This work is part of my PhD thesis at the University of Warwick under the supervision of John Greenlees. I am extremely grateful to John for his guidance, comments and ideas on this project. I would like in particular to thank Julian Lawrence Demeio for many useful conversations, suggestions and ideas especially on the Algebraic Geometry side.

\section{Elliptic cohomology of CP(V)}
\label{sec:1}
Recall that $G=\T^2$, and $H_1=1\times \T$ and $H_2=\T\times 1$ are the two privileged subgroups. We have also the quotient $\bar G:=G/H_1\cong H_2\cong \T$ as group of equivariance. We have fixed an elliptic curve $\c$ over $\C$ together with a coordinate $t_e\in \O_{\c,e}$. This gives us a rational $\bar G$-equivariant elliptic cohomology theory $\ecgbar\in \A(\bar G)$ \eqref{eq:ecgbar}, as well as a rational $G$-equivariant elliptic cohomology theory $\ecg\in \A(G)$ (Theorem \ref{thm:fundconstr2}). We have also the complex abelian surface $\X=\X_G=\c\times \c$ associated to $G$ and $\c$.

Given a $\bar G$ complex representation $V$ the main goal of the paper is to compute the reduced cohomology of the pointed space:
\[
\ecgbar^*(\cp V)
\]
where $\cp V$ is the $\bar G$-space of complex lines in  $V$. First of all notice the isomorphism of $\bar G$-spaces:
\begin{equation}
	\label{eq:cpv2}
	\cp V\cong S(V\otimes_{\C} w)/H_1
\end{equation}
where $w$ is the natural complex representation of $H_1$ and $S(V\otimes_{\C} w)$ is the $\bar G$-space of vectors of unit norm in the complex vector space $V\otimes_{\C} w$. Notice that $V\otimes_{\C} w$ is a complex representation of $G$ of the same dimension of $V$, where $H_1$ acts on the second factor of the tensor product, while $\bar G$ acts on the first one. The computation is made possible since $\ecgbar$ and $\ecg$ are $H_1$-split:

\begin{theorem}
	\label{thm:split}
	Let $\ecg$ be $G$-elliptic cohomology (Theorem \ref{thm:fundconstr2}), and $\ecgbar$ be $\bar G$-elliptic cohomology \eqref{eq:ecgbar}. Then there is a natural transformation of $G$-cohomology theories
	\[
	\epsilon: \inflat \ecgbar \longrightarrow \ecg
	\]
	which induces an isomorphism
	\begin{equation}
		\label{eq:equivalence}
		[\quotient GH_+, \inflat \ecgbar]^G_*\cong [\quotient GH_+, \ecg]^G_*
	\end{equation}
	for every subgroup $H$ of $G$ such that $H\cap H_1=\{1\}$.
\end{theorem}
We are going to prove this Theorem in the remaining sections of this paper, in particular in Section \ref{sec:cpvmap} we build the map $\epsilon$ while in Section \ref{sec:cpvsplit} we prove the $H$-equivalence \eqref{eq:equivalence}. In the interest of clarity we now present how to compute elliptic cohomology of $\cp V$ starting from $H_1$-splitness.
 
We have an immediate Corollary:
\begin{corollary}
	\label{cor:h1quot}
	For any $H_1$-free $G$-space $X$:
	\begin{equation}
		\ecg^*(X)\cong \ecgbar^*(X/H_1)
	\end{equation}
\end{corollary}
\begin{proof}
	Since $X$ is $H_1$-free, it is built using cells $G/H_+$ with $H\cap H_1=\{1\}$, for which we have the equivalence \eqref{eq:equivalence}. Therefore:
	\[
	\ecg^*(X)=[X, \ecg]^G_*\cong [X,\inflat \ecgbar]^G_*\cong [X/H_1, \ecgbar]^{\bar G}_*=\ecgbar^*(X/H_1)
	\]
	where the second isomorphism is the orbits-inflation adjunction.
\end{proof}

Applying this corollary to  \eqref{eq:cpv2} allows us to reduce the computation of $\bar G$-equivariant elliptic cohomology to a computation of $G$-equivariant elliptic cohomology:
\begin{equation}
	\label{eq:gequiv}
	\ecgbar^* (\cp V_+)\cong \ecgbar^*(\quotient {S(V\otimes w)_+}{H_1})\cong \ecg^*(S(V\otimes w)_+).
\end{equation}
\begin{rem}
	Notice we have to add a disjoint basepoint to $S(V\otimes w)$, leading us to compute $\ecgbar^* (\cp V_+)$. To obtain $\cp V$ without the added basepoint notice that stably:
	\begin{equation*}
		\cp V_+ \cong \cp V \vee S^0
	\end{equation*}
	and therefore
	\begin{equation}
		\label{eq:basepoint}
		\ecgbar^* (\cp V_+)\cong \ecgbar^*(\cp V)\oplus \ecgbar^*(S^0)\cong \ecgbar^*(\cp V) \oplus \C,
	\end{equation}
	since by \cite[Theorem 1.1]{john:elliptic} we have
	\begin{equation*}
		\ecgbar^k(S^0)\cong 
		\begin{cases*}
			H^0(\c, \O_\c)\cong \C & $k$ \text{even} \\
			H^1(\c, \O_\c)\cong \C & $k$ \text{odd}.
		\end{cases*}
	\end{equation*}
\end{rem}

$G$-equivariant elliptic cohomology of $S(V\otimes w)_+$ is easier to compute in view of the cofibre sequence of $G$-spaces:
\begin{equation}
	S(V\otimes w)_+ \longrightarrow S^0 \longrightarrow S^{V\otimes w}
\end{equation}
inducing a long exact sequence:
\begin{equation}
	\label{eq:leselliptic}
	\ecg^*(S^{V\otimes w}) \longrightarrow \ecg^*(S^0) \longrightarrow  \ecg^*(S(V\otimes w)_+).
\end{equation}
The first two terms of \eqref{eq:leselliptic} are computed in \eqref{eq:valuetheory}, therefore we can deduce the third term from kernel and cokernel of the first map.

To find the associated divisor $D_{V\otimes w}$ \eqref{eq:assdiv} decompose $V$ as a sum of one dimensional complex representations $V=\bigoplus_n\alpha_nz^n$ where $z$ is the natural representation of $\bar G$ and $\alpha_n \geq 0$. Notice $z^n\otimes w$ is a one dimensional representation of $G$ with kernel the connected codimension $1$ subgroup:
\[
H_{d_n}:=\Ker(z^n\otimes w)=\{(x,y)\in G=\T\times \T\mid x^ny=1\}
\]
and corresponding divisor:
\begin{equation}
	\label{eq:ddn}
	D_{d_n}=\XF(H_{d_n})=\{(P,Q)\in \X =\c\times \c\mid nP+Q=e\}.
\end{equation}
Therefore $V\otimes w=\bigoplus_n\alpha_n(z^n\otimes w)$ is a complex representation of $G$ of the same dimension as $V$, whose associated divisor is:
\begin{equation}
	\label{eq:divisor}
	D_{V\otimes w}=\sum_n\alpha_nD_{d_n}.
\end{equation}
We have two distinct proves depending if the divisor \eqref{eq:divisor} has only one coefficient different from zero, or more than one coefficient different from zero.

When $V$ has only one isotypic component, then $\cp V\cong \Cp^{\alpha-1}$ with trivial $\T$-action. We can prove part (1) of Theorem \ref{thm:iso2}:
\begin{lemma}
	\label{thm:iso}
	If $V$ has only one isotypic component: $V=\alpha z^n$ with $\alpha\geq 1$ then
	\begin{equation}
		\ecgbar^k(\cp V)\cong \C^{\alpha-1}
	\end{equation}
	for every $k\in \Z$.
\end{lemma}
\begin{proof}
	Combining \eqref{eq:gequiv} and \eqref{eq:leselliptic} we only need to understand kernel and cokernel of the map
	\begin{equation}
		\label{eq:ellcohmap}
		\ecg^*(S^{V\otimes w}) \longrightarrow \ecg^*(S^0).
	\end{equation}
	By \eqref{eq:divisor} the divisor associated to $V\otimes w$ is $D_{V\otimes w}=\alpha D_{d_n}$, denote $D:=D_{d_n}$ the smooth irreducible curve. By \eqref{eq:valuetheory} it is enough to understand kernel and cokernel for the map
	\begin{equation}
		\label{eq:cohmap}
		H^*(\X,\O(-\alpha D)) \longrightarrow H^*(\X,\O_{\X})
	\end{equation}
	induced by the inclusion of sheaves $\O(-\alpha D)\hookrightarrow \O_{\X}$. We compute this map for the various degrees.
	
	In degree zero $H^0(\X,\O(-\alpha D))=0$, since $\X$ is a compact complex abelian surface, therefore a function regular on all the surface is constant and since it has a zero at $D$ it is the constant zero. Therefore \eqref{eq:cohmap} in degree zero has zero kernel and cokernel $\C$.
	
	In degree $2$ by Serre duality 
	\begin{equation}
		\label{eq:h2dual}
		\hiso 2{-\alpha D}\cong \hiso 0{\alpha D}^{\vee}\cong \C^{\alpha}
	\end{equation}
	where the second isomorphism is obtained as follows. Consider the divisor $D'=\alpha (e)$ on the single elliptic curve $\c$, defining the line bundle $\L:=\O_{\c}(\alpha (e))$. The projective morphism $f:=\pi_{d_n}: \X \to \c$ is such that $f^*\L\cong \O(\alpha D)$, and $f_*\O_{\X}\cong \O_\c$ by \cite[Exercise 3.12 Chapter 5]{Liu}. As a consequence
	\begin{equation}
		\label{eq:doubledirect}
		f_*f^*\L=f_*(\O_{\X}\otimes_{\O_{\X}}f^*\L)\cong f_*(\O_{\X})\otimes_{\O_\c} \L \cong \O_\c \otimes_{\O_\c} \L\cong \L
	\end{equation}
	by the projection formula \cite[exercise 5.1]{hartshorne}. Therefore
	\begin{equation}
		\label{eq:directimage}
		\hiso 0{\alpha D}\cong  H^0(\X, f^*\L) \cong H^0(\c, f_*f^*\L)\cong H^0(\c, \L) \cong \C^{\alpha},
	\end{equation}
	where the second isomorphism is due to the equality of functors $\Gamma(\X,\_) = \Gamma(\c, f_*(\_))$, while the last one is Riemann-Roch for the elliptic curve $\c$.
	Moreover the map
	\[
	\C\cong H^0(\X,\O_{\X}) \rightarrowtail \hiso 0{\alpha D} \cong \C^{\alpha}
	\]
	is injective being the inclusion of the constant functions. By Serre duality the dual map, \eqref{eq:cohmap} in degree 2, is surjective with kernel $\C^{\alpha-1}$.
	
	In degree 1 by Serre duality 
	\[
	H^1(\X,\O(-\alpha D))\cong H^1(\X,\O(\alpha D))^\vee \cong \C^\alpha
	\] 
	since by Riemann-Roch \cite[Theorem I.12]{beauville}
	\begin{equation}
		h^0(\X,\alpha D)-h^1(\X,\alpha D)=\frac 12 \alpha^2(D.D)=0.
	\end{equation}
	We can see $(D.D)=0$ since $D$ is linearly equivalent to the translated $D+\lambda$, which is an irreducible curve disjoint from $D$, or alternatively simply using the genus formula \cite[Theorem I.15]{beauville} since $D$ is irreducible of genus 1 being isomorphic to the elliptic curve $\c$. By Lemma \ref{lem:imh1} the image of the map \eqref{eq:cohmap} in degree 1 is precisely $\hiso 1{-D}\cong \C$, combining this with the computations \eqref{eq:cohmap} we obtain that the kernel in degree 1 of \eqref{eq:cohmap} is $\C^{\alpha-1}$ while the cokernel is $\C$.
	
	In conclusion we obtain that $\ecgbar^k(\cp V_+)$ has in even degrees the degree zero cokernel and the degree one kernel of \eqref{eq:cohmap}, therefore is isomorphic to $\C^\alpha$. In odd degrees we have the degree 1 cokernel and the degree 2 kernel of \eqref{eq:cohmap}, so $\C^{\alpha}$ as well. By \eqref{eq:basepoint} $\ecgbar^k(\cp V)\cong \C^{\alpha-1}$ for every $k\in \Z$.
\end{proof}
\begin{lemma}
	\label{lem:imh1}
	Let $D=D_{d_n}$ be any of the smooth divisors \eqref{eq:ddn}. For every integer $\alpha\geq 1$ the image of the map
	\begin{equation}
		\hiso 1{-\alpha D} \to H^1(\X,\O_{\X})
	\end{equation}
	is precisely $\hiso 1{-D} \subset H^1(\X,\O_{\X})$.
\end{lemma}
\begin{proof}
	When $\alpha=1$ we have that $D$ is a subvariety of $\X$ whose ideal sheaf is precisely $\O_{\X}(-D)$, and structure sheaf $\O_D$ that we can see as a sheaf on $\X$ via the inclusion map $\iota: D \to \X$. We have the closed subscheme short exact sequence \cite[14.3.B]{vakil:notes}:
	\begin{equation}
		\label{eq:subsexact}
		0 \to \O(-D) \to \O_{\X} \to \iota_*\O_D \to 0
	\end{equation}
	inducing a long exact sequence in cohomology. Since $H^0(\X,\O_{\X})\xrightarrow{\cong} H^0(\X,\iota_*\O_D)\cong \C$ then $\hiso 1{-D}\to H^1(\X,\O_{\X})$ is injective.
	
	When $\alpha \geq 2$ we can tensor the exact sequence \eqref{eq:subsexact} with the invertible sheaf $\O(-(\alpha-1)D)$ obtaining the exact sequence of sheaves:
	\begin{equation}
		\label{eq:sesiter}
		0 \to \O(-\alpha D) \to \O(-(\alpha-1)D) \to \iota_*\O_D(-(\alpha-1)D) \to 0
	\end{equation}
	inducing a long exact sequence in cohomology. By the adjunction formula \cite[Theorem 1.37]{Liu}:
	\begin{equation*}
		(\omega_{\X}\otimes \O(D))\restriction_D =\omega_D
	\end{equation*}
	where in our case $\omega_{\X}\cong \O_{\X}$ and $\omega_D\cong \O_D$, resulting in $\iota_*\O_D(D)\cong \iota_*\O_D$ and by induction $\iota_*\O_D(\beta D)\cong \iota_*\O_D$ for every integer $\beta$.
	
	Therefore we have that the last term of \eqref{eq:sesiter} is $\iota_*\O_D$, and we can analyse the degree 2 piece of the long exact sequence induced:
	\begin{equation}
		\label{eq:lesiterated}
		H^1(\iota_*\O_D) \to H^2(\O(-\alpha D)) \to H^2(\O(-(\alpha-1)D)) \to H^2(\iota_*\O_D)=0.
	\end{equation}
	The first map in \eqref{eq:lesiterated} is injective simply by dimension computation: $H^1(\iota_*\O_D)\cong \C$, and the other two terms are computed in \eqref{eq:h2dual}. Consequently the map $H^1(\O(-\alpha D)) \twoheadrightarrow H^1(\O(-(\alpha-1)D))$ is surjective. In conclusion we have the chain of maps:
	\[
	H^1(\O(-\alpha D))\twoheadrightarrow H^1(\O(-(\alpha-1)D)) \twoheadrightarrow \dots \twoheadrightarrow H^1(\O(-D)) \rightarrowtail H^1(\O_{\X})
	\]
	where all the maps are surjective except the last one that is injective, giving us the desired result.
\end{proof}

When $V$ has more than one isotypic component, then the proof of part (2) of Theorem \ref{thm:iso2} has a different approach:
\begin{theorem}
	\label{thm:mainthm}
	If $V$ has more than one isotypic component: $V=\bigoplus_n\alpha_n z^n$ with $\alpha_n\geq 0$ then
	\begin{equation}
		\ecgbar^k(\cp V)\cong 
		\begin{cases*}
			0 & $k$ even \\
			\C^d & $k$ odd.
		\end{cases*}
	\end{equation}
	where $d=\sum_{i<j}\alpha_i\alpha_j(i-j)^2$.
\end{theorem}

\begin{proof}
	Exactly as in the proof of Theorem \ref{thm:iso} we only need to understand kernel and cokernel of the map
	\begin{equation}
		\label{eq:cohmap2}
		H^*(\X,\O(-D)) \longrightarrow H^*(\X,\O_{\X})
	\end{equation}
	where $D:=D_{V\otimes w}=\sum_n\alpha_nD_{d_n}$ is the divisor associated to $V\otimes w$ \eqref{eq:divisor}.
	
	When $V$ has more than one isotypic components, the associated divisor $D$ is an ample divisor. To show this we can use the Nakai-Moishezon criterion \cite[Theorem V.1.10]{hartshorne}. This criterion states that $D$ is an ample divisor on $\X$ if and only if $D.D>0$ and $D.\c'>0$ for every irreducible curve $\c'$ on $\X$. First let us prove that $D_{d_r}.D_{d_s}=(r-s)^2$ for any two integers $r$ and $s$. If $r=s$ then $D_{d_r}.D_{d_r}=0$ since $D_{d_r}$ is linearly equivalent to the translated $D_{d_r}+\lambda$, which is an irreducible curve disjoint from $D_{d_r}$, or alternatively simply using the genus formula \cite[Theorem I.15]{beauville} since $D_{d_r}$ is irreducible of genus 1 being isomorphic to the elliptic curve $\c$. If $r\neq s$ than the two curves $D_{d_r}$ and $D_{d_s}$ are transverse in each point of intersection therefore we simply need to count the intersection points \cite[Definition I.3]{beauville}. There is an autoisogeny of $\X$ bringing $D_{d_r}$ to $D_{d_0}=D_2=\{(P,Q)\in \X \mid Q=e\}$. Under this isogeny $D_{d_s}$ is brought to $D_{d_{s-r}}=\{(P,Q)\in \X \mid (s-r)P+Q=e\}$. The intersection of these last two curves is easily computed to be $\c[s-r]$ which has cardinality $(s-r)^2$. Therefore:
	\begin{equation}
		\label{eq:symbilin}
		D.D=\sum_{i,j\in \Z}\alpha_i\alpha_j(D_{d_i}.D_{d_j})=2\sum_{i<j}\alpha_i\alpha_j(i-j)^2>0
	\end{equation}
	since there are at least two different integers $r\neq s$ such that $\alpha_r, \alpha_s >0$ ($V$ has more than one isotypic component). It remains to show that for every irreducible curve $\c'$ in $\X$, the curve $\c'$ intersects either $D_{d_r}$ or $D_{d_s}$, so that $D.\c'>0$. If $\c'\cap D_{d_r}=\emptyset$ than $\c'$ is necessarily parallel to $D_{d_r}$, meaning $\c'=D_{d_r}+\lambda$ is a translated of $D_{d_r}$, but all these curves intersect $D_{d_s}$ since $s\neq r$.
	
	The sheaf $\O(D)$ is invertible and ample, so we can use Kodaira vanishing theorem \cite[Remark III.7.15]{hartshorne} obtaining $\hiso iD=0$ for $i\geq 1$. Therefore only the zeroth cohomology is nonzero and we can compute it with Riemann-Roch \cite[Theorem I.12]{beauville}:
	\[
	h^0(\X,D)=\frac 12(D.D)=\sum_{i<j}\alpha_i\alpha_j(i-j)^2=d
	\]
	where $(D.D)$ is computed in \eqref{eq:symbilin}. Moreover the map 
	\[
	\C\cong H^0(\X,\O_{\X}) \rightarrowtail \hiso 0D \cong \C^d
	\]
	is injective being the inclusion of the constant functions. By Serre duality the dual map, \eqref{eq:cohmap2} in degree 2, is surjective with kernel $\C^{d-1}$.
	
	In conclusion we obtain that $\ecgbar^k(\cp V_+)$ has in even degrees the degree zero cokernel of \eqref{eq:cohmap2}: $H^0(\X,\O_{\X}) \cong  \C$. In odd degrees we have the degree 1 cokernel and the degree 2 kernel of \eqref{eq:cohmap2}, so $\C^{d+1}$. By \eqref{eq:basepoint} $\ecgbar^k(\cp V)$ has zero in even degrees and $\C^d$ in odd degrees.
\end{proof}

\section{The circle case revisited}
\label{sec:cpvcirclecase}
We revise the construction of circle equivariant elliptic cohomology from \cite{john:elliptic}. We begin by recalling the definition of the TP-topology for the single elliptic curve \cite[Definition 7.1]{john:elliptic}:
\begin{definition}
	\label{def:tptopc}
	Over the set $\c$ define the torsion point topology $\tp \c$ with generating closed subsets $\{\c\gen n\}_{n\geq 1}$, where $\c\gen n$ are the elements of exact order $n$ in $\c$.
\end{definition}
This is a ringed topological space with the pushforward of the structure sheaf: $\tp \O_{\c}$. Likewise \cite[Definition 8.2]{john:elliptic} choose a coordinate for $\c$ at $e$:
\begin{choice}
	\label{choice:te}
	Choose $t_e\in \tp \O_{\c,e}\subset \O_{\c,e}$ vanishing to the first order at $e$.
\end{choice}

Since we are working in characteristic zero there is a unique logarithm for the formal group law of $\c$ (see for example \cite[Proposition 3.1]{strickland:formal}), namely a strict isomorphism with the additive formal group law: 
\[
\hatn e=f(t_e)\in (\tp {\O_{\c,e}})^\wedge _{m_e}\cong \C[[t_e]],
\]
where $m_e<\tp {\O_{\c,e}}$ is the ideal of those functions vanishing at $e$.

Denoting $[n]:\c \to \c$ the multiplication by $n$ map in the elliptic curve, we can pullback coordinate for the various $\c\gen n$:
\begin{definition}
	For every integer $n\geq 1$ define the coordinate and completed coordinate:
	\begin{equation}
		\begin{split}
			t_n &:=[n]^*(t_e)\in \tp {\oce} \\
			\hatn n &:=[n]^*(\hatn e)=f(t_n)\in (\tp {\O_{\c,e}})^\wedge _{m_e}
		\end{split}
	\end{equation}
\end{definition}
\begin{rem}
	Notice $t_n\in \ocn n$ since it has a zero of degree one on all points of $\c[n]$, so we can use it as a coordinate for the irreducible closed subset $\c\gen n$. In the same way $\hatn n$ is an element in the completed ring $(\ocn n)^\wedge$.
\end{rem}
\begin{definition}
	The fundamental rings are the stalks in the TP-topology:
	\begin{equation}
		\label{eq:kt}
		\begin{split}
			\kt &:=\tp {\O_{\c, \eta(\c)}}=\{f\in \K(\zar \c)\mid f \; \text{has poles only at points of finite order of} \; \c\} \\
			\ocn n &:= \tp {\O_{\c, \eta(\c\gen n)}}= \{f\in \kt\mid f \; \text{is regular at} \; \c \gen n\}
		\end{split}
	\end{equation}
\end{definition}

Denote $\bar \F$ the family of finite subgroups of $\bar G$: for every $n\geq 1$ we have the cyclic subgroup $C_n$ of order $n$. The fundamental ring for the algebraic models $\A(\bar G)$ is
\[
\ofbar =\prod_{n\geq 1} \hbgbar n= \prod_{n\geq 1} \Q[c_n]
\]
where $c_n\in \hbgbar n$ of degree $-2$ is the Euler class \eqref{eq:eulerclass} of a character having kernel $C_n$.
\begin{definition}
	Define the torsion injective $\ofbar$-module
	\begin{equation}
		\label{eq:ttt}
		\ttt:=\bigoplus_{n\geq 1}\overline {\kt/\ocn n}
	\end{equation}
	where the action is defined on the $n$-th component as follows. The Euler class $c_n$ acts as $\hatn n$:
	\begin{equation*}
		c_n\cdot [f]=[\hatn n \cdot f]\in \kt/\ocn n.
	\end{equation*}
	Notice the action is well defined since $t_n$ vanishes at first order at $\c\gen n$, so that powers of $\hatn n$ do not contribute after a certain integer and the sum is finite.
\end{definition}

\begin{definition}
	Define the graded surjective $\ofbar$-module map
	\begin{equation}
		\label{eq:q}
		q:\egbarinv \ofbar \otimes \overline \kt \twoheadrightarrow \ttt
	\end{equation}
	as follows. On the $n$-th component on pure tensor elements:
	\[
	c_n^k\otimes f \xmapsto{q} [\hatn n^k\cdot f].
	\]
	Notice this is well defined also for negative powers of the euler class simply considering the inverse power series $\hatn n ^{-1}$.
\end{definition}
\begin{definition}
	Let $\nt:=\Ker(q)$ be the kernel of the map $q$ \eqref{eq:q}. We have the exact sequence of $\ofbar$-module:
	\begin{equation}
		\label{eq:exacts1}
		\begin{tikzcd}
			\nt \arrow[r,tail] & \egbarinv \ofbar \otimes \overline \kt \arrow[r, two heads, "q"] & \ttt.
		\end{tikzcd}
	\end{equation}
	Define the algebraic model for circle-equivariant elliptic cohomology $\ecgbar \in \A(\bar G)$ to be the object:
	\begin{equation}
		\label{eq:ecgbar}
		\ecgbar :=
		\left[ 
		\begin{tikzcd}
			\egbarinv \ofbar \otimes \overline \kt \\
			\nt \arrow[u, tail]
		\end{tikzcd}
		\right]
		\in \A(\bar G)
	\end{equation}
	with structure map the natural inclusion of \eqref{eq:exacts1}.
\end{definition}
\begin{rem}
	Since $\ttt$ is torsion injective we have immediately the injective resolution of $\ecgbar$ in $\A(\bar G)$:
	\begin{equation*}
		\left[ 
		\begin{tikzcd}
			\egbarinv \ofbar \otimes \overline \kt \\
			\nt \arrow[u]
		\end{tikzcd}
		\right]
		\rightarrowtail
		\left[ 
		\begin{tikzcd}
			\egbarinv \ofbar \otimes \overline \kt \\
			\egbarinv \ofbar \otimes \overline \kt \arrow[u]
		\end{tikzcd}
		\right]
		\twoheadrightarrow
		\left[ 
		\begin{tikzcd}
			0 \\
			\ttt \arrow[u]
		\end{tikzcd}
		\right].
	\end{equation*}
\end{rem}
In this way the object \eqref{eq:ecgbar} we have defined here has the same values on spheres of complex representations as the object constructed in \cite[Theorem 1.1]{john:elliptic}. The analogy with the construction of $\T^2$-equivariant elliptic cohomology of \cite[Theorem 1.4]{barucco:t2} is more transparent now. Namely the Cousin complex for $\tp {\O_{\c}}$ is the following short exact sequence of sheaves that we can find in \cite[Corollary 9.3]{john:elliptic}:
\begin{equation}
	\label{eq:ccsingle}
	\tp {\O_{\c}} \rightarrowtail \i {\c}(\kt) \twoheadrightarrow \bigoplus_{n\geq 1}\i {\c\gen n}(\quotient {\kt}{\ocn n})
\end{equation}
where we denote $\iota_Z(M)$ the sheaf of constant value $M$ on the closed subset $Z$. Moreover \eqref{eq:ccsingle} is a flabby resolution of the structure sheaf that we can use to link the values of $\ecgbar$ on spheres of complex representations with the appropriate cohomology of the associated line bundle on $\c$ \cite[Theorem 1.1]{john:elliptic}.

\section{Building the map}
\label{sec:cpvmap}
The aim of this section is to build the map $\epsilon: \inflat \ecgbar \to \ecg$ of Theorem \ref{thm:split}. The first step is to identify the functor $\inflat:\A(\bar G) \to \A(G)$ and we do this in general for tori of any rank.

\subsection{The inflation functor}
Only for this subsection $G=\T^r$ is a generic torus of some rank $r$, we fix a connected subgroup $K$ and we define the quotient group $\bar G:=G/K$, with quotient map $q:G \to \bar G$. For a generic connected subgroup $H$ of $G$ denote $L:=\gen {H,K}$ the subgroup generated by $H$ and $K$, and $\bar H:=q(L)$ the image subgroup of $H$ in $\bar G$.
\begin{proposition}
	\label{prop:inflatedsystem}
	For a $\bar G$-spectrum $\bar X$, the value of $\pai {\A(G)}(\inflat \bar X)$ \eqref{eq:paistarfunctor} at a connected subgroup $H$ is given by:
	\begin{equation}
		\phi^H(\inf {\bar G}G \bar X)\cong \oef H \tens {\oef L} \phi^{\bar H}(\bar X),
	\end{equation}
	with structure maps induced by the structure maps of $\pai {\A(\bar G)}(\bar X)$.
\end{proposition}

To prove this result we will need some lemmas about $G$-spectra that holds more generally for compact Lie groups (with the appropriate assumptions of subgroups being normal). Moreover these results do not rely on the fact that we have rationalized everything and therefore holds also if we do not localize at $S^0\Q$.

Given a $G$-spectrum $X$ we can pick a strictly increasing indexing sequence $V_1 \subsetneq V_2 \subsetneq \dots \subsetneq \U$ for the complete $G$-universe $\U$, and use the canonical presentation of $X$ \cite[pg. 12]{kervaire}, namely there is a weak equivalence of $G$-spectra
\begin{equation}
	\label{eq:canpres}
	X\simeq \hocolim_{V_n} \susp {-V_n}\sus X(V_n)
\end{equation}
where the homotopy colimit is taken over the indexing sequence $\{V_n\}$ and $\sus X(V_n)$ is the suspension spectrum of the $G$-space $X(V_n)$. In general we will prove that two spectra $X$ and $Y$ are weakly equivalent proving that for every $V_n$ in an indexing sequence the two $G$-spaces $X(V_n)$ and $Y(V_n)$ are isomorphic.
\begin{definition}
	If $H\subseteq G$ is a subgroup of $G$, we say that the indexing sequence $\{V_n\}$ is an $H$-separated indexing sequence, if the sequence of $H$-fixed points $V_1^H\subsetneq V_2^H \subsetneq \dots \subseteq \U^H$ is a strictly increasing indexing sequence for the $G/H$-universe $\U^H$.
\end{definition}
When this happens we have a nice description for the geometric fixed point:
\begin{equation}
	\label{eq:geomfixed}
	\fix H X \simeq \hocolim_{V_n^H}\susp {-V_n^H}\sus ((X(V_n))^H).
\end{equation}
In this sense $(\fix H X)(V_n^H)=(X(V_n))^H$ as $G/H$-spaces.
\begin{lemma}
	Given a $\bar G$-spectrum $\bar X$ and a $K$-separated indexing sequence $\{V_n\}$ the values of \eqref{eq:canpres} for the inflated spectrum are:
	\begin{equation}
		(\inf {\bar G}G\bar X)(V_n) = \susp {V_n-V_n^K}\inf {\bar G}G(\bar X(V_n^K))
	\end{equation}
	where $V_n-V_n^K$ is the orthogonal complement of $V_n^K$.
\end{lemma}
\begin{proof}
	We prove the statement for suspension spectra, and using \eqref{eq:canpres} it holds in general for all spectra. If $\bar S$ is a $\bar G$-space:
	\begin{equation}
		\begin{split}
			(\sus \inflat \bar S)(V_n) &= \susp {V_n}\inflat \bar S = \susp {V_n-V_n^K} \susp {V_n^K}(\inflat \bar S)= \susp {V_n-V_n^K} \inflat (\susp {V_n^K} \bar S) =\\
			&= \susp {V_n-V_n^K} \inflat ((\sus \bar S)(V_n^K)).
		\end{split}
	\end{equation}
	The first equality is the definition of suspension spectrum, the third equality is because $V_n^K$ is a $G/K$-representation and commutes with inflation, while the last one is again the definition of suspension spectrum. 
\end{proof}
\begin{lemma}
	for every $\bar G$-spectrum $\bar X$ we have a weak equivalence of $G/H$-spectra:
	\begin{equation}
		\label{eq:fixpointsinflation}
		\fix H(\inf {\bar{G}}G\bar X)\simeq \inf {\bar{G}/\bar{H}}{G/H}(\fix {\bar{H}}\bar X)
	\end{equation}
\end{lemma}
\begin{proof}
	Denote $L:=\gen {H,K}$ the subgroup generated by $H$ and $K$ in $G$, and notice that the inflation on the left hand side of \eqref{eq:fixpointsinflation} makes sense since $\bar G/\bar H\cong (G/K)/(L/K)\cong G/L\cong (G/H)/(L/H)$. Choose an indexing sequence $\{V_n\}$ for the complete $G$-universe $\U$ that is $H$-separated and $K$-separated, and such that the indexing sequence $\{V_n^H\}$ is $L/H$-separated (It is an indexing sequence for the group $G/H$). We will show that the two spectra \eqref{eq:fixpointsinflation} have isomorphic values on the indexing sequence $\{V_n^H\}$.
	\begin{equation}
		\label{eq:chain1}
		\begin{split}
			(\fix H(\inf {\bar{G}}G\bar X))(V_n^H) &= ((\inflat \bar X)(V_n))^H= (\susp {V_n-V_n^K} (\inflat (\bar X(V_n^K))))^H=  \\
			&= \susp {V_n^H-V_n^L} (\inflat (\bar X(V_n^K)))^H.
		\end{split}
	\end{equation}
	The first isomorphism is \eqref{eq:geomfixed}, while the second one is \eqref{eq:fixpointsinflation}. For the right hand side:
	\begin{equation}
		\label{eq:chain2}
		(\inf {\bar{G}/\bar{H}}{G/H}(\fix {\bar{H}}\bar X))(V_n^H) = \susp {V_n^H-V_n^L} \inf {\bar{G}/\bar{H}}{G/H} ((\fix {\bar H} \bar X)(V_n^L))= \susp {V_n^H-V_n^L} \inf {\bar{G}/\bar{H}}{G/H} ((\bar X(V_n^K))^{\bar H}).
	\end{equation}
	The $G/H$-spaces \eqref{eq:chain1} and \eqref{eq:chain2} are isomorphic since \eqref{eq:fixpointsinflation} is an isomorphism for the $\bar G$-space $\bar X(V_n^K)$:
	\begin{equation*}
		(\inflat (\bar X(V_n^K)))^H = \inf {\bar{G}/\bar{H}}{G/H} ((\bar X(V_n^K))^{\bar H})	
	\end{equation*}
\end{proof}
\begin{proof}[Proof of Proposition \ref{prop:inflatedsystem}]
	Let us first prove the case when $H=\{1\}$ is the trivial subgroup, by \eqref{eq:paistarfunctor} we need to prove the isomorphism:
	\begin{equation}
		\label{eq:trivialcase}
		\pai G(\inflat (\bar X) \wedge  DE\F_+)\cong \O_{\F}\tens {\oef K} \pai {\bar G}(\bar X \wedge DE\F/K_+).
	\end{equation}
	The proof works exactly as the proof of \cite[Lemma 9.2]{john:torusI}. Both sides of \eqref{eq:trivialcase} are homology theories of $\bar X$ (for the right hand side of \eqref{eq:trivialcase}: by \cite[Corollary 5.7]{john:torusI} the ring $\of$ is flat over $\oef K$ so that tensoring with it is an exact functor) and we have a natural transformation of homology theories induced by:
	\begin{equation*}
		[S^0, \bar X\wedge DE\F/K_+]^{\bar G}_* \to [\inf {\bar G}G S^0, \inf {\bar G}G (\bar X\wedge DE\F/K_+)]^G_* \to [S^0, (\inf {\bar G}G\bar X)\wedge \defp]^G_*.
	\end{equation*}
	The first map is induced by the inflation map, and the second one is induced by a map of $G$-spectra 
	$\inf {\bar G}G \defhp K \to \defp$ described in \cite[pag.22]{john:torusI}. We only need to prove that this natural transformation of homology theories is an isomorphism for the various cells of $\bar G$. When $\bar X=S^0$ is the sphere spectrum \eqref{eq:trivialcase} holds since by \cite[Theorem 7.4]{john:torusI}: $\pai G(\defp)=\of$. For more general cells of $\bar G$ the isomorphism  \eqref{eq:trivialcase} follows from the case of the sphere spectrum by the ``Rep($G$)-iso argument'' \cite[Theorem 11.2]{john:torusI} since both homology theories satisfy Thom isomorphism (smashing with $\defp$ gives Thom isomorphism by \cite[Corollary 8.5]{john:torusI}).
	
	For the case when $H$ is an arbitrary connected subgroup of $G$ we need to prove the more general isomorphism:
	\begin{equation}
		\label{eq:generalh}
		\pai {G/H}(\fix H(\inf {\bar G}G \bar X) \wedge \defhp H)\cong \oef H \tens {\oef L} \pai {\bar G/\bar H}(\fix {\bar H}\bar X\wedge \defhp {\bar H}).
	\end{equation}
	We can reduce it to the case of the trivial subgroup since by \eqref{eq:fixpointsinflation} the left hand side of \eqref{eq:generalh} becomes
	\begin{equation*}
		\pai {G/H}(\fix H(\inf {\bar G}G \bar X) \wedge \defhp H)\cong \pai {G/H}(\inf {\bar{G}/\bar{H}}{G/H}(\fix {\bar{H}}\bar X) \wedge \defhp H).
	\end{equation*}
	With this substitution \eqref{eq:generalh} is precisely \eqref{eq:trivialcase} for the new ambient group $G/H$, with inflation along the quotient map $G/H\to G/L$ and the $G/L$-spectrum $\fix {\bar{H}}\bar X$.
\end{proof}

\subsection{Building the map}
We can now return to our case $G=\T^2$ and $\bar G=G/H_1$, and apply proposition \ref{prop:inflatedsystem} to $\ecgbar\in \A(\bar G)$ \eqref{eq:ecgbar} to explicitly obtain
\begin{equation}
	\label{eq:inflated}
	\inflat(\ecgbar) =\left[
	\begin{tikzcd}
		\eginv \of \otimes \overline \kt & \\
		\ehinv 1 \of \tens 1 \nt \arrow[u] & \ehinv i \of \otimes \overline \kt \arrow[ul]  \\
		\of \tens 1 \nt \arrow[u] \arrow[ur] &
	\end{tikzcd}
	\right]
\end{equation}
where in codimension 1 (middle row) the value at $H_1$ (middle row left) is the only one different from the values at all the other connected codimension one subgroups $H_i$ with $i\neq 1$ (middle row right).

We can also explicitly compute $\ecg\in \A(G)$ \cite[Definition 5.2]{barucco:t2} as the kernel of the map $\phi_0$ \cite[Lemma 5.26]{barucco:t2}. 
\begin{definition}
	\label{def:ni2}
	For every connected codimension 1 subgroup $H_i$ of $G$ define $N_i$ to be the kernel of the surjective $\oefh i$-map $\phi_0^i$ \cite[Lemma 5.26]{barucco:t2}:
	\begin{equation}
		\label{eq:exactecghi}
		\begin{tikzcd}
			N_i \arrow[r, tail] & \eghinv i \oefh i \otimes \overline \kg \arrow[r, two heads, "\phi_0^i"] & T_i.
		\end{tikzcd}
	\end{equation}
	Where $\K$ and $T_i$ are respectively defined in \cite[(3.8), (5.18)]{barucco:t2}. 
\end{definition}
Then we can explicitly compute $\ecg$:
\begin{equation}
	\label{eq:ecghom}
	\ecg =\left[
	\begin{tikzcd}
		\eginv \of \otimes \overline \kg\\
		\ehinv i \of \tens i N_i \arrow[u, tail] \\
		\Ker(\phi_0(1)) \arrow[u,tail]
	\end{tikzcd}
	\right]
\end{equation}

\begin{definition}
	\label{def:epsilon}
	Define the map $\epsilon: \inflat \ecgbar \longrightarrow \ecg$ in $\A(G)$ to be the map that at the vertex:
	\begin{equation}
		\label{eq:epsilon}
		\phi^G(\epsilon):=\pi_1^*: \overline \kt \to \overline \kg
	\end{equation}
	is the graded map that has the pullback $\pi_1^*$ in each even degree, where $\pi_1:\X\to \c$ is the projection having $H_1$ as kernel.
\end{definition}
\begin{lemma}
	\label{lem:epsilon}
	The map \eqref{eq:epsilon} extends to a well defined map $\epsilon: \inflat \ecgbar \longrightarrow \ecg$ in $\A(G)$.
\end{lemma}
\begin{proof}
	Notice that the structure maps of $\ecg$ \eqref{eq:ecghom} are all injective, therefore \eqref{eq:epsilon} determines the map $\epsilon$ at each level. We only need to verify that for each connected subgroup the target is the correct one.
	
	For the subgroup $H_1$, the induced map  $\phi^{H_1}(\epsilon):\nt \to N_1$ is the only map that completes the diagram:
	\begin{equation}
		\label{eq:diagramh1}
		\begin{tikzcd}
			\nt \arrow[r, tail] \arrow[d, dashrightarrow, "\phi^{H_1}(\epsilon)"] &
			\eghinv 1\oefh 1 \otimes \overline \kt \arrow[r, twoheadrightarrow, "q"] \arrow[d, "\Id \otimes \pi_1^*"] & 
			\ttt \arrow[d, "\pi_1^*"]\\
			N_1 \arrow[r, rightarrowtail] &
			\eghinv 1\oefh 1 \otimes \overline \kg  \arrow[r, twoheadrightarrow, "\phi_0^1"] &
			T_1
		\end{tikzcd}
	\end{equation}
	where the top row is the exact sequence \eqref{eq:exacts1} whose pullback $\pi_1^*$ gives the bottom row which is the exact sequence \eqref{eq:exactecghi} for the subgroup $H_1$. The right vertical map is induced by $\pi_1^*$ and therefore  the right square of \eqref{eq:diagramh1} commutes since the coordinates $\hat t_{1,j}$ \cite[Definition 3.23]{barucco:t2} used in $\phi_0^1$ satisfy:
	\begin{equation*}
			\hat t_{1,j} =(\pi_1^j)^*(\hatn e)=\pi_1^*\circ [j]^*(\hatn e)=\pi_1^*(\hatn j).
	\end{equation*}
	
	For all the other connected codimension one subgroups $H_i$ with $i\neq 1$ to show the existence of the induced  dotted arrow $\phi^{H_i}(\epsilon)$:
	\begin{equation}
		\label{eq:diagramhi}
		\begin{tikzcd}
			\oefh i\otimes \overline \kt \arrow[r, rightarrowtail] \arrow [d, dashrightarrow, "\phi^{H_i}(\epsilon)"] & 
			\eghinv i \oefh i\otimes \overline \kt  \arrow[d, "\Id \otimes \pi_1^*"] & \\
			N_i \arrow[r, rightarrowtail] &
			\eghinv i \oefh i\otimes \overline \kg  \arrow[r, twoheadrightarrow, "\phi_0^i"] & T_i
		\end{tikzcd}
	\end{equation}
	we only need to verify that the zig-zag from the top left to the bottom right is the zero map, since the bottom row is exact. The map $\phi_0^i$ is zero on every element in $\oefh i \otimes \im(\pi_1^*)$  since in the image of $\pi_1^*$ there are only meromorphic functions with poles at $D_{1j}$ \cite[Definition 2.10]{barucco:t2}, so they are regular on all $D_{ij}$ with $i\neq 1$.
	
	For the trivial subgroup to show the existence of the induced dotted arrow $\phi^1(\epsilon)$:
	\begin{equation}
		\label{eq:diagram1}
		\begin{tikzcd}
			\of \tens 1 \nt \arrow[r, rightarrowtail] \arrow [d, dashrightarrow, "\phi^1(\epsilon)"] & 
			\eginv \of \otimes \overline \kt  \arrow[d, "\Id \otimes \pi_1^*"] & \\
			\Ker(\phi_0(1)) \arrow[r, rightarrowtail] & 
			\eginv \of \otimes \overline \kg  \arrow[r, "\phi_0(1)"] &
			\bigoplus_{i\geq 1}\ehinv i\of \tens i T_i
		\end{tikzcd}
	\end{equation}
	we need to verify that the zig-zag from the top left corner to the bottom right corner is zero. The bottom row of \eqref{eq:diagram1} is the beginning of the trivial subgroup level of the injective resolution \eqref{eq:injres4prime} of $\ecg$. This zig-zag is indeed zero because of the previous diagrams. For the component $i=1$ the map is zero because the same zig-zag in \eqref{eq:diagramh1} is zero. In the same way for all the other components $i\neq 1$ the map is zero because the same zig-zag in \eqref{eq:diagramhi} is zero.
\end{proof}

\section{Proving the H-equivalence}
\label{sec:cpvsplit}
To finish the proof of Theorem \ref{thm:split} we are left to show that for every subgroup $H$ of $G$ such that $H\cap H_1=1$ the map $\epsilon$ (Definition \ref{def:epsilon}) induces an isomorphism
\begin{equation}
	\label{eq:isomsplit}
	[\quotient GH_+, \inflat \ecgbar]^G_* \cong [\quotient GH_+, \ecg]^G_*.
\end{equation}
To do so we will use the Adams spectral sequence \eqref{eq:ass2} for the homology functor $\pi_*^\A$, and show that $\epsilon$ induces an isomorphism of the second page of the two Adams spectral sequences:
\begin{equation}
	\label{eq:assiso}
	\Ext^{*,*}_{\A}(\pi_*^\A(\quotient GH_+),\inflat \ecgbar) \cong  
	\Ext^{*,*}_{\A}(\pi_*^\A(\quotient GH_+),\ecg).
\end{equation}

The first step is to identify the algebraic model for the natural cells $G/H_+$ which is done in \ref{sec:naturalcells}. We turn then in building an injective resolution for $\inflat \ecgbar$ in \ref{sec:injectiveresolutiont} to compute the $\Ext$ groups in Theorem \ref{thm:extgroupiso}. We will also recall the injective resolution of $\ecg$ . 

\subsection{Algebraic model for natural cells}
\label{sec:naturalcells}
We explicitly compute the algebraic model for the natural cells: $\pigh i$ for the connected codimension one subgroups $H_i$ of $G$, and $\pigf$ for the finite subgroups $F$. 
\begin{lemma}
	\label{lem:naturalcells}
	For every $i\geq 1$ the algebraic model $\pigh i$ has values:
	\begin{equation}
		\label{eq:phihi}
		\phi^{H_i}(\pigh i)=\Sigma {\hbg {H_i}}/{e(z_i)}=\Sigma \Q
	\end{equation}
	\begin{equation}
		\label{eq:phi1f}
		\phi^1(\pigh i) =\bigoplus_{F \subseteq H_i} \Sigma {\hbg F}/{e(z_i)}
	\end{equation}
	while $\phi^{G}(\pigh i)=\phi^{H_j}(\pigh i)=0$ for $j\neq i$. The structure maps are the ones induced by the suspended sphere $\pai \A(S^{z_i})$ \eqref{eq:structuremaps}. Here $e(z_i)$ is the Euler class \eqref{eq:eulerclass} of a character $z_i$ of $G$ having $H_i$ as kernel.
\end{lemma}
\begin{proof}
	We start from the cofibre sequence of $G$-spaces:
	\begin{equation}
		\gh i\longrightarrow S^0 \xrightarrow{e(z_i)} S^{z_i}.
	\end{equation}
	Applying the suspension functor we obtain a cofibre sequence in $G$-spectra, which induces a long exact sequence for the homology functor $\pai \A$:
	\begin{equation}
		\label{eq:zihom}
		\pigh i \longrightarrow \pai \A(S^0) \xrightarrow{e(z_i)} \pai \A(S^{z_i}).
	\end{equation}
	
	By \eqref{eq:sphere} the induced map $e(z_i)$ in \eqref{eq:zihom} at the levels of the subgroups $G$ and $H_j$ with $j\neq i$ is the identity since $z_i^G=z_i^{H_j}=0$, therefore we obtain zero as kernel and cokernel at those levels. At the level of the subgroup $H_i$ since $z_i^{H_i}=z_i$ the map $e(z_i)$ is induced by the map
	\begin{equation}
		\label{eq:phihi1}
		\phi^{H_i} (e(z_i)): \oefh i \to \susp {z_i}\oefh i
	\end{equation}
	that sends the unit $\iota\in \oefh i$ to precisely the Euler class $e(z_i^{H_i})\in \oefh i$ as defined in \eqref{eq:eulerclass}. Using the coordinate $c_i$ \eqref{eq:cij}, the Euler class $e(z_i^{H_i})$ has $c_i$ in the $H_i$-th component and $1$ in all the other components. Therefore the map \eqref{eq:phihi1} is injective so it has no kernel, while the cokernel is $\susp 2 \Q[c_i]/(c_i)$, which is the only contribution to \eqref{eq:phihi}.
	
	At the bottom level we can apply the same argument, the map:
	\begin{equation}
		\label{eq:phi1f1}
		\phi^{1} (e(z_i)): \of \to \susp {z_i}\of
	\end{equation}
	sends the unit $\iota \in \of$ to precisely the Euler class $e(z_i)\in \of$, which by \eqref{eq:eulerclass} and \eqref{eq:xi} has components:
	\begin{equation}
		\label{eq:ezi}
		e(z_i)_F=
		\begin{cases*}
			1 & if $F\nsubseteq H_i$ \\
			x_i & if $F\subseteq H_i$
		\end{cases*}
	\end{equation}
	Therefore the map \eqref{eq:phi1f1} is injective and the cokernel is precisely the suspension of \eqref{eq:phi1f}.
\end{proof}
\begin{lemma}
	\label{lem:gf}
	For every finite subgroup $F$ of $G$ the algebraic model $\pigf$ has a zero at each subgroup level except at the bottom level \eqref{eq:valuet2} where it has the value:
	\begin{equation}
		\label{eq:pigf}
		\pigf(1)=\bigoplus_{ F' \subseteq F} \Sigma^2\hbg {F'}/(\xa, \xb)= \bigoplus_{ F' \subseteq F} \Sigma^2 \Q
	\end{equation}
	the suspended rationalized Burnside ring of $F$. We are using the coordinates $\hbg {F'}\cong \Q[x_A,x_B]$ \eqref{eq:xaxb}. 
\end{lemma}
\begin{proof}
	In this case we use directly the definition of the homology functor $\pai \A$ \eqref{eq:paistarfunctor}. If $H$ is a connected subgroup of $G$ different from the trivial one, then $\phi^H(\pigf)=0$, since $\fix H(\gf)=0$ because the $H$-fixed points of the $G$-space $\gf$ is only the basepoint.
	
	At the bottom level we have the rationalized Burnside ring of $F$ \cite[Definition 1.3.5]{barnes:rational}:
	\begin{equation}
		\label{eq:burnsidering}
		\pai G(\defp \wedge G/F_+)\cong \pai G(G/F_+)\cong \susp 2A(F)\cong \bigoplus_{F'\subseteq F}\susp 2\Q.
	\end{equation}
	For the first isomorphism of \eqref{eq:burnsidering} we use that natural cells are strongly dualizable, more precisely by \cite[(4.16)]{johnmay:equivariant}:
	\begin{equation}
		D(\gf)=F(\gf, S^0)\simeq S^{-L(F)}\wedge \gf
	\end{equation}
	where $L(F)$ is the tangent $F$-representation at the identity coset $G/F$ and $D(G/F_+)$ is the dual of $G/F_+$. Therefore:
	\begin{equation}
		\begin{split}
			\pai G(\defp \wedge \gf) &\cong \pai G (\defp \wedge S^{L(F)} \wedge D(\gf))\\  
			&\cong \pai G(S^{L(F)}\wedge D(\efp \wedge \gf)) \\
			&\cong \pai G (S^{L(F)} \wedge D(\gf))\cong \pai G(\gf).
		\end{split}
	\end{equation}
	The weak equivalence $\efp \wedge \gf \simeq \gf$ can be obtained from the isotropy separation cofibre sequence of $\gf$:
	\begin{equation}
		\efp \wedge \gf \to \gf \to \etf \wedge \gf
	\end{equation}
	noticing the last term is nullhomotopic since $\gf$ has stable isotropy in $\F$.
	The second isomorphism of \eqref{eq:burnsidering} can be obtained as follows, using the restriction-coinduction adjunction:
	\begin{equation*}
		\begin{split}
			[S^0,\gf]^G &\cong [S^0,S^{L(F)} \wedge F(\gf,S^0)]^G \\
			&\cong [S^{-L(F)}, F_F(G_+,S^0)]^G \\
			&\cong [\res FS^{-L(F)}, S^0]^F\cong \Sigma^2[S^0,S^0]^F=\susp 2 A(F)
		\end{split}
	\end{equation*}
\end{proof}

\subsection{Injective resolution of $\inflat \ecgbar$}
\label{sec:injectiveresolutiont}
We build an injective resolution of $\inflat \ecgbar$ \eqref{eq:inflated} in the abelian category $\A(G)$:
\begin{equation}
	\label{eq:infres}
	\begin{tikzcd}
		0 \arrow[r] &\inflat \ecgbar \arrow[r] & \inj_0' \arrow[r, "\phi_0'"] & \inj_1' \arrow[r, "\phi_1'"] &  \inj_2' \arrow[r] & 0.
	\end{tikzcd}
\end{equation}
We first rewrite $\inflat \ecgbar$ in a more convenient form:
\begin{equation}
	\label{eq:ectinfl}
	\inflat \ecgbar =\left[
	\begin{tikzcd}
		\eginv \of \tens 1 \nt \\
		\ehinv i \of \tens 1 \nt \arrow[u, rightarrowtail] \\
		\of \tens 1 \nt \arrow[u, rightarrowtail]
	\end{tikzcd}
	\right]
\end{equation}
where at each level we are tensoring the $\of$-module of $\pi_*^\A(S^0)$ (example \ref{ex:modelsphere}) with the nub $\nt$ \eqref{eq:ecgbar} of $\ecgbar$:

\begin{rem}
	\label{rem:changering}
	To obtain \eqref{eq:ectinfl} from \eqref{eq:inflated} simply notice that when $i\neq 1$ we have $\egh 1 \subseteq \eh i \subseteq \eg$ and therefore:
	\[
	\ehinv i \of \tens 1 \nt \cong \ehinv i \of \tens 1 \eghinv 1 \nt \cong \ehinv i \of \tens 1 (\eghinv 1 \oefh 1 \otimes \K_\T) \cong \ehinv i \of \otimes \K_\T.
	\]
	The same is true for $\eginv \of \tens 1 \nt \cong \eginv \of \otimes \K_\T$. 
\end{rem}

Notice that for the family of all subgroups of $G$ the universal space is precisely $E[All]_+\simeq S^0$. Therefore combining \cite[12.3]{john:torusI} and \cite[10.2]{john:torusI} we obtain the injective resolution of $S^0$ in $\A(G)$: 
\begin{equation*}
	S^0 \longrightarrow f_G(\Q) \longrightarrow \bigoplus_{i\geq 1} f_{H_i}(\bigoplus_{j\geq 1} \Sigma^2 \hombg {H_i^j}) \longrightarrow  f_1(\bigoplus_{F\in \F} \Sigma^4 \hombg F)\longrightarrow 0.
\end{equation*}
We can rewrite this last sequence in a more convenient form for us:
\begin{equation}
	\label{eq:s0res}
	S^0 \longrightarrow f_G(\Q) \longrightarrow \bigoplus_{i\geq 1} f_{H_i}(\frac {\eghinv i \oefh i}{\oefh i}) \longrightarrow  f_1(M)\longrightarrow 0,
\end{equation}
where $M$ is the $\of$-module:
\begin{equation}
	\label{eq:m1}
	M:= \quotient {(\bigoplus_{i\geq 1}\frac {\eginv \of}{\ehinv i\of})}{\Delta}
\end{equation}
and $\Delta$ is the diagonal submodule: the image of the map 
\[
\eginv \of \longrightarrow \bigoplus_{i\geq 1}\frac {\eginv \of}{\ehinv i\of}.
\]
Since $\nt$ is flat over $\oefh 1$  \cite[Lemma 5.3]{john:elliptic}, we can simply tensor over this ring every $\of$-module of \eqref{eq:s0res} to obtain an injective resolution of $\inflat \ecgbar$ in $\A(G)$. More precisely the terms of the injective resolution \eqref{eq:infres} are:
\begin{equation}
	\label{eq:injectivesprime}
	\begin{split}
		\inj_0' &=f_G(\overline \kt)\\
		\inj_1' &= f_{H_1}(\ttt) \bigoplus_{i>1}f_{H_i}(\frac {\eghinv i \oefh i}{\oefh i} \otimes \overline \kt)\\
		\inj_2' &= f_1(M\tens 1 \nt),
	\end{split}
\end{equation}
where $M$ is defined in \eqref{eq:m1}, $\nt$ in \eqref{eq:exacts1}, and we have used Remark \ref{rem:changering}. Notice we obtain $f_{H_1}(\ttt)$ in $\inj_1'$ since by \cite[Proposition 3.1]{john:torusII} the module $\ehinv 1\of$ is flat over $\oefh 1$. Moreover $\inj_1'$ is indeed injective by \ref{cor:inj1} and \ref{cor:inj2} since $\ttt$ is a torsion injective $\oefh 1$-module and
\[
\Sigma^2 \hombg {H_i^j}  \otimes \overline \kt
\]
is a torsion injective $H^*(B\quotient G{H_i^j})$-module.

\subsection{Injective resolution of $\ecg$}
\label{sec:injectiveresolutiong}
For convenience we recall here the injective resolution of $\ecg$ in $\A(G)$ built in \cite[Lemma 5.3]{barucco:t2}.
\begin{equation}
	\label{eq:injres4prime}
	\begin{tikzcd}
		0 \arrow[r] &\ecg \arrow[r] & \inj_0 \arrow[r, "\phi_0"] & \inj_1 \arrow[r, "\phi_1"] & \inj_2 \arrow[r, "0"] & 0
	\end{tikzcd}
\end{equation}
where the terms are:
\begin{equation}
	\label{eq:injectives}
	\begin{split}
		\inj_0 &:=f_G(\overline \kg)\\
		\inj_1 &:= \bigoplus_{i\geq 1}f_{H_i}(T_i)\\
		\inj_2 &:=f_1(N),
	\end{split}
\end{equation}
with
\begin{equation}
	\label{eq:net}
	\begin{split}
		N &:= \bigoplus_{F}\overline {\h 2F}\\
		T_i &:= \bigoplus_{j\geq 1} \overline {\K/\odij}
	\end{split}
\end{equation}
as defined in \cite[(5.18), (5.23)]{barucco:t2}, with all local cohomology modules coming from the Cousin complex of the structure sheaf $\tp {\O_{\X}}$ \cite[(4.33)]{barucco:t2}. 

\subsection{Computing the Ext groups}
We are finally ready to finish the proof of Theorem \ref{thm:split}, namely to show the isomorphism of the $\Ext$ groups in the Adams spectral sequences \eqref{eq:isomsplit}. To do so we will need some technical results that we prove after the main Theorem, and that we use in the proof.

\begin{theorem}
	\label{thm:extgroupiso}
	For every subgroup $H$ of $G$ such that $H\cap H_1=\{1\}$ the map $\epsilon$ \eqref{eq:epsilon} induces an isomorphism between the terms in the Adams spectral sequence:
	\begin{equation}
		\Ext_{\A}^{*,*}(\quotient GH_+, \inflat(\ecgbar))\cong \Ext_{\A}^{*,*}(\quotient GH_+, \ecg)
	\end{equation}
\end{theorem}
\begin{proof}
	First recall the injective resolutions of $\inflat \ecgbar$ \eqref{eq:infres}, and $\ecg$ \eqref{eq:injres4prime}, as well as the algebraic model for the natural cells (Lemma \ref{lem:naturalcells}, and Lemma \ref{lem:gf}).
	
	It is enough to show that for every such $H$, in the following commutative diagram (obtained taking morphisms into the injective resolutions):
	\begin{equation}
		\label{eq:commdiaghom}
		\begin{tikzcd}
			\Hom_{\A}(G/H_+,\inj_0') \arrow[r] \arrow[d] &
			\Hom_{\A}(G/H_+,\inj_1') \arrow[r,"\delta_1'"] \arrow[d, "\epsilon_1"] &
			\Hom_{\A}(G/H_+,\inj_2') \arrow[r] \arrow[d, "\epsilon_2"] &
			0\\
			\Hom_{\A}(G/H_+,\inj_0)  \arrow[r] &
			\Hom_{\A}(G/H_+,\inj_1)  \arrow[r,  "\delta_1"] &
			\Hom_{\A}(G/H_+,\inj_2)  \arrow[r] & 
			0
		\end{tikzcd}
	\end{equation}
	the vertical maps (induced by $\epsilon$) are all isomorphisms. We will first show this when $H$ has codimension 1, and then when $H$ is finite. Recall the definition of the injective  objects in the first row \eqref{eq:injectivesprime}, and in the second one \eqref{eq:injectives}. 
	
	If $H$ has codimension 1, then necessarily $H$ is connected and without loss of generality we can suppose $H=H_2$. The first column of \eqref{eq:commdiaghom} is zero since by the adjunction \eqref{eq:adjunction}:
	\begin{equation*}
		\Hom_{\A}(\gh 2, f_G(\overline{\K_\T}))\cong \Hom_{\Q}(0, \overline{\K_\T})=0
	\end{equation*}
	and with the same argument: $\Hom_{\A}(\gh 2, f_G(\overline{\K}))=0$.
	
	Moving to the second column, by Lemma \ref{lem:naturalcells} the object $\gh 2$ has a zero at each subgroup $H_s$ with $s\neq 2$. Therefore the only summands of $\inj_1'$ and $\inj_1$ that do contribute are the ones constant below $H_2$, using the coordinate $\hbg {H_2}\cong \Q[c_2]$ \eqref{eq:cij} by the adjunction \eqref{eq:adjunction} we obtain:
	\begin{equation*}
		\begin{split}
			\Hom_{\A}(\gh 2,\inj_1') &\cong \Hom_{\A}(\gh 2, f_{H_2}(\frac {\eghinv 2 \oefh 2}{\oefh 2}\otimes \overline{\K_\T})) \\
			&\cong \Hom_{\Q[c_2]}(\Sigma \Q, \quotient {\Q[c_2^{\pm1}]}{\Q[c_2]}\otimes \overline{\K_\T})\cong \Sigma \overline{\K_\T}.
		\end{split}
	\end{equation*}
	Explicitly the morphism $f\in \overline{\K_\T}$ is the one that sends the unit to the element $[c_2^{-1}]\otimes f$. In the same way:
	\begin{equation*}
		\begin{split}
			\Hom_{\A}(\gh 2,\inj_1) &\cong \Hom_{\A}(\gh 2, f_{H_2}(T_2)) \\
			&\cong \Hom_{\Q[c_2]}(\Sigma \Q, \overline{\quotient {\K}{\O_{D_2}}})\cong \Sigma \overline{\quotient {\O_{D_2}}{m_2}}
		\end{split}
	\end{equation*}
	where $m_2$ is the ideal of $\O_{D_2}$ of those functions vanishing at $D_2$ and explicitly the morphism $[f]\in \overline{\quotient {\O_{D_2}}{m_2}}$ is the one that sends the unit to the element $[\hat t_2^{-1}f]\in \overline{\quotient {\K}{\O_{D_2}}}$. 
	
	By Lemma \ref{lem:previousarticle} the map $\epsilon_1$ \eqref{eq:commdiaghom} is an isomorphism, since in each even degree it is the map:
	\[
	\pi_1^*:\kt \xrightarrow{\cong} \quotient {\O_{D_2}}{m_2}.
	\]
	
	Moving to the third column of \eqref{eq:commdiaghom}, using  the adjunction \eqref{eq:adjunction}:
	\begin{equation}
		\label{eq:homgh2i2}
		\begin{split}
			\Hom_{\A}(\gh 2,\inj_2') &\cong \Hom_{\A}(\gh 2, f_{1}(M\tens 1 \nt)) \\
			&\cong \Hom_{\of}(\bigoplus_{F \subseteq H_2}\Sigma  {\hbg F}/{(x_2)}, M\tens 1 \nt)\\
			&\cong \bigoplus_{F \subseteq H_2}\Hom_{\hbg F}(\Sigma  {\hbg F}/{(x_2)}, (M\tens 1 \nt)_F)\\
			&\cong \bigoplus_{n_1\geq 1} \Sigma \overline{\quotient {\kt}{\O_{\c\gen {n_1}}}}. 
		\end{split}
	\end{equation}
	where $M$ is defined in \eqref{eq:m1}, $\nt$ in \eqref{eq:exacts1} and $x_2$ in \eqref{eq:xi}. The last isomorphism of \eqref{eq:homgh2i2} is obtained as follows. First for every $n_1\geq 1$ we have the finite cyclic subgroup $F=C_{n_1}=H_1^{n_1}\cap H_2\subseteq H_2$. By Lemma \ref{lem:comfyform} an element in the $F$-th component of the third row of \eqref{eq:homgh2i2} is a map sending the unit of $\Sigma  {\hbg F}/{(x_2)}$ to a sum of elements of the form \eqref{eq:comfyelements} with $h=1$. Therefore the $F$-th component of the map $\delta_1'$ \eqref{eq:commdiaghom}:
	\begin{equation}
		\label{eq:delta1f}
		(\delta_1')_F:\Sigma \overline{\K_\T} \to \Hom_{\hbg F}(\Sigma  {\hbg F}/{(x_2)}, (M\tens 1 \nt)_F)
	\end{equation}
	is surjective with kernel $\Sigma \overline{\O_{\c\gen {n_1}}}$, since \eqref{eq:delta1f} sends $f$ to the map which sends the unit to the element
	\[
	[(0, [x_1^{-k}x_2^{-1}], 0, \dots)]\tens {\Q[x_1]}(x_1^k\otimes f)\in (M\tens 1 \nt)_F
	\]
	with $ft_{n_1}^k$ regular on $\c\gen {n_1}$. Therefore $(\delta_1')_F$ induces the last isomorphism of \eqref{eq:homgh2i2}. In the same way:
	\begin{equation}
		\label{eq:gh2i2}
		\begin{split}
			\Hom_{\A}(\gh 2,\inj_2) &\cong \bigoplus_{F \subseteq H_2}\Hom_{\hbg F}(\Sigma  {\hbg F}/{(x_2)}, N_F)\\ 
			&\cong \bigoplus_{F \subseteq H_2} \Sigma \overline{\quotient {({\O_{D_2}}/{m_2})}{({\O_{F}}/{m_2})}}.
		\end{split}
	\end{equation}
	where by \cite[(4.32)]{barucco:t2} the $F$-th component of $N$ \eqref{eq:net} can be described as:
	\begin{equation}
		\label{eq:nf}
		N_F= \overline {\h 2F}\cong\overline{\bigoplus_{i\geq 1}(\quotient {\kg}{\O_{\din i}})/\K},
	\end{equation}
	and $\O_F$ is defined in \cite[(3.10)]{barucco:t2}. The last isomorphism of \eqref{eq:gh2i2} is obtained as before. 
	The $F$-th component of the map $\delta_1$ \eqref{eq:gh2i2}:
	\begin{equation}
		\label{eq:delta1f2}
		(\delta_1)_F: \Sigma \overline{\quotient {\O_{D_2}}{m_2}} \to \Hom_{\hbg F}(\Sigma  {\hbg F}/{(x_2)}, N_F)
	\end{equation}
	is surjective, with kernel $\Sigma \overline {{\O_{F}}/{m_2}}$. This is because by \cite[Lemma 5.49]{barucco:t2} every element in \eqref{eq:nf} admits a representative $[(0,[g],0,\dots)]$ with only the second component different from zero. Therefore an element in the $F$-th component of \eqref{eq:gh2i2} is a map sending the unit to an element of the form: $[(0,[\hat{t}_2^{-1}f], 0,\dots)]$ with $[f]\in \overline{\quotient {\O_{D_2}}{m_2}}$, which is precisely the $F$-th component of $\delta_1([f])$. 
	This shows that $(\delta_1)_F$ induces the last isomorphism of \eqref{eq:gh2i2}. 
	
	By Lemma \ref{lem:previousarticle} the map $\epsilon_2$ \eqref{eq:commdiaghom} is an isomorphism, since in every $F$-th component and in each even degree it is the map:
	\[
	\pi_1^*:\quotient {\kt}{\O_{\c\gen {n_1}}} \xrightarrow{\cong} \quotient {({\O_{D_2}}/{m_2})}{({\O_{F}}/{m_2})}.
	\]
	where $F=H_1^{n_1}\cap H_2$.
	
	Let us prove the case when $H$ is a finite subgroup of $G$ such that $H\cap H_1=\{1\}$. By Lemma \ref{lem:finitesub} we have $H\cong C_n$ is cyclic and without loss of generality $H\subseteq H_2$. Notice that $H=H_1^{n}\cap H_2$. Using the adjunction \eqref{eq:adjunction}, the first two columns of \eqref{eq:commdiaghom} are zero, since by Lemma \ref{lem:gf} the algebraic model for $G/H_+$ is zero everywhere except at the trivial subgroup. Therefore we are only left to prove that the map $\epsilon_2$ \eqref{eq:commdiaghom} is an isomorphism. By the adjunction \eqref{eq:adjunction}:
	\begin{equation}
		\label{eq:homagf}
		\Hom_{\A}(G/H_+,\inj_2') \cong \bigoplus_{F\subseteq H} \Hom_{\hbg F}(\Sigma^2 \Q, (M\tens 1 \nt)_F) \cong \bigoplus_{n_1\mid n} \overline{\quotient {\O_{\c\gen {n_1}}}{m}},
	\end{equation} 
	where $m<\O_{\c\gen {n_1}}$ is the ideal of those functions vanishing at $\c\gen {n_1}$. The last isomorphism of \eqref{eq:homagf} can be proven similarly to \eqref{eq:homgh2i2} term by term for every $F=H_1^{n_1}\cap H_2$. An element in the $F$-th component of \eqref{eq:homagf} is an $\hbg {F}\cong \Q[x_1,x_2]$-module map sending the unit of $\Q$ to an element of $(M\tens 1 \nt)_{F}$ which is zero if multiplied by $x_1$ or $x_2$ \eqref{eq:xi}. By Lemma \ref{lem:comfyform} such an element is sum of elements of the form:
	\[
	[([x_1^{-1}x_2^{-1}], 0, 0, \dots)]\tens {\Q[x_1]}(x_1 \otimes \quotient f{t_{n_1}})
	\]
	with $[f]\in \overline{\quotient {\O_{\c\gen {n_1}}}{m}}$, hence we obtain \eqref{eq:homagf}. Similarly:
	\begin{equation}
		\label{eq:homagf2}
		\Hom_{\A}(G/H_+,\inj_2) \cong \bigoplus_{F\subseteq H} \Hom_{\hbg {F}}(\Sigma^2 \Q, N_{F}) \cong \bigoplus_{F\subseteq H} \overline{\quotient {\O_{F}}{\gen {m_1,m_2}}},
	\end{equation} 
	where $m_i<\O_{F}$ is the ideal of those functions vanishing at $D_{i,n_i}$. The second isomorphism of \eqref{eq:homagf2} is proven for every $F$ similarly to \eqref{eq:gh2i2}. The element $[f]\in \overline{\quotient {\O_{F}}{\gen {m_1,m_2}}}$ defines the $\hbg {F}\cong \Q[x_1,x_2]$-module map that sends the unit of $\Q$ to the following element with only the first component different from zero:
	\[
	[([\frac {f}{\hattin 1 \hat t_{2,1}}], 0, 0, \dots)] \in N_{F},
	\]
	where we have used \eqref{eq:nf}.
	
	By Lemma \ref{lem:previousarticle} the map $\epsilon_2$ \eqref{eq:commdiaghom} is an isomorphism, since in every $F$-th component and in each even degree it is the map:
	\[
	\pi_1^*:\quotient {\O_{\c\gen {n_1}}}{m} \xrightarrow{\cong} \quotient {\O_{F}}{\gen {m_1,m_2}}
	\]
	where $\pi_1^{-1}(\c\gen {n_1})= D_{1,n_1}$, so that a function in $m$ has pullback in $m_1$.
\end{proof}
\begin{rem}
	\label{rem:dictionary}
	The form of this proof suggests a geometric counterpart of the statement directly in Algebraic Geometry. The Lie group homomorphism $z_1:G \twoheadrightarrow \bar G$ induces the projection $\pi_1=\XF(z_1):\X \twoheadrightarrow \c$, while for every subgroup $H$ such that $H\cap H_1=1$ we have the inclusion $i_H:H \hookrightarrow G$ inducing the immersion $\iota=\XF(i_H):\XF(H) \hookrightarrow \X$. This induces the following dictionary between Topology and Algebraic geometry:
	\begin{center}
		\begin{tabular}{cc}
			\toprule
			Topology & Algebraic Geometry \\
			\midrule
			$\ecg$ & $\O_{\X}$ \\
			$\ecgbar$ & $\O_{\c}$  \\
			$\inflat \ecgbar$ & $\pi_1^*\O_{\c}$  \\
			$S^W$ & $\O_{\X}(D_W)$ \\
			$S(V\otimes w)_+$ & $\Sigma^{-1}\quotient{\O_{\X}(D_{V\otimes w})}{\O_{\X}}$ \\
			$\cp V_+$ & $(\pi_1)_*(\Sigma^{-1}\quotient{\O_{\X}(D_{V\otimes w})}{\O_{\X}})$ \\
			$i_H^*\inflat \ecgbar \simeq i_H^*\ecg$ & $\iota^*\pi_1^*\O_{\c}\cong \iota^*\O_{\X}$ \\
			\bottomrule
		\end{tabular}
	\end{center}
	where $V$ is a $\T$-representation and $W$ is a $G$-representation. Denoting $\F=\Sigma^{-1}\quotient{\O_{\X}(D_{V\otimes w})}{\O_{\X}}$, we can also translate the statement of Corollary \ref{cor:h1quot} (using homology instead of cohomology):
	\[
	\pi_*^G(S(V\otimes w)_+\wedge \ecg) \cong \pi_*^{\bar G}(\cp V_+\wedge \ecgbar)
	\] 
	into its Algebraic Geometry counterpart:
	\[
	H^*(\X,\F \otimes \O_{\X}) \cong H^*(\c,(\pi_1)_*(\F \otimes \O_{\X}))\cong H^*(\c, (\pi_1)_*\F \otimes \O_{\c}),
	\]
	that possibly can be proven directly using geometric arguments like \eqref{eq:doubledirect}, \eqref{eq:directimage}, and working on the higher images functors.
\end{rem}
We conclude with the results needed for the proof of Theorem \ref{thm:extgroupiso}.
\begin{lemma}
	\label{lem:finitesub}
	If $F\subseteq G$ is a finite subgroup of $G$ such that $F\cap H_1=\{1\}$, then $F$ is cyclic and there is a connected codimension $1$ subgroup $H_i$ such that $F\subseteq H_i$ and $H_i\cap H_1=\{1\}$.
\end{lemma}
\begin{proof}
	Suppose $F$ is not cyclic, then it contains at least a $p$-group of the form $F'\cong \Z/p\Z \times \Z/p\Z$ for a prime $p$. But inside $G$ there is only one copy of such a $p$-group, namely the subgroup of elements of order $p$ in $G$: $G[p]=\T[p]\times \T[p]$. This is immediate to see since every element in $F'$ has order $p$, therefore $F'\subseteq G[p]$ and they have the same cardinality. This gives us a contradiction since  $|F'\cap H_1|=p>1$.
	
	To prove the existence of such an $H_i$ let us think about $G$ as the quotient $G=\T\times \T\cong (\R\times \R)/(\Z\times \Z)$. Every connected codimension 1 subgroup $H_i$ is determined by a point $P_i=(\lambda_i,\mu_i)\in \Z \times \Z$ for a pair of coprime integers $\lambda_i$ and $\mu_i$, namely $H_i$ is the image in the quotient $(\R\times \R)/(\Z\times \Z)$ of the line in $\R\times \R$ connecting the origin to the point $P_i$. For example $H_1$ is determined by the point $P_1=(0,1)$. By the first part of the Lemma $F$ is cyclic and generated by a point $Q=[(a/n,b/n)]\in (\R\times \R)/(\Z\times \Z)$ with $a,b \in \Z$ and where $n$ is the order of $F$. Since $F\cap H_1=\{1\}$ we have $a\neq 0$ and that $a$ and $n$ are coprimes, otherwise we can find a non-trivial multiple of $Q$ not in $\Z\times \Z$ but that lies in $H_1$. Therefore we can find $r,s\in \Z$ such that $rn+sa=1$. For obvious reason $s$ and $n$ are coprimes, so the $s$-th multiple of $Q$: 
	\[
	Q'=sQ=[(\frac {sa}n,\frac {sb}n)]=[(r+\frac {sa}n,\frac {sb}n)]=[(\frac 1n, \frac {sb}n)]
	\]
	is a generator for the subgroup $F$. The subgroup $H_i$ defined by the point $P_i=(1,sb)$ satisfies all the requirements. We have $H_i\cap H_1=\{1\}$ since the line connecting the origin to $P_i$ does not intersect any integer vertical line. Moreover the line connecting the origin to $P_i$ passes through $Q'$, therefore $F\subseteq  H_i$ since $Q'$ generates $F$.  
\end{proof}

The following is an extension of \cite[Lemma 5.49]{barucco:t2}, with the same notation.
\begin{lemma}
	\label{lem:previousarticle}
	The pullback along the projection $\pi_1:\X \to \c$ induces an isomorphism:
	\begin{equation}
		\label{eq:pi11}
		\pi_1^*: \kt \xrightarrow{\cong} \quotient {\O_{D_2}}{m_2},
	\end{equation}
	that restricted to $\O_{\c\gen {n_1}}$ for any integer $n_1\geq 1$ gives an isomorphism:
	\begin{equation}
		\label{eq:pi12}
		\pi_1^*: \O_{\c\gen {n_1}} \xrightarrow{\cong} \quotient {\O_{F}}{m_2}
	\end{equation}
	where $F= H_1^{n_1}\cap H_2$.
\end{lemma}
\begin{proof}
	The isomorphism \eqref{eq:pi11} is \cite[(5.50)]{barucco:t2} for the point $P=\{e\}$ and for the projection $\pi_1$ instead of $\pi_2$. To prove \eqref{eq:pi12} it is enough to notice that $\pi_1^{-1}(\c\gen {n_1})=D_{1,n_1}$ and therefore \eqref{eq:pi11} restricted to $\O_{\c\gen {n_1}}$ gives functions that are regular also at $D_{1,n_1}$, by \cite[Lemma 5.49]{barucco:t2} we obtain that the image is precisely $\O_{F}$.
\end{proof}
\begin{lemma}
	\label{lem:comfyform}
	For every finite subgroup $F=H_1^{n_1}\cap H_2\subseteq H_2$, every element in the $F$-th component $(M\tens 1 \nt)_F$ is sum of elements of the form
	\begin{equation}
		\label{eq:comfyelements}
		[([x_1^{-k}x_2^{-h}], 0, 0, \dots)]\tens {\Q[x_1]}(x_1^k\otimes f)\in (M\tens 1 \nt)_F
	\end{equation}
	such that $h,k\geq 1$ and $ft_{n_1}^k\in \O_{\c\gen {n_1}}$. Here $M$ is defined in \eqref{eq:m1}, $\nt$ in \eqref{eq:exacts1}, and we have used the coordinates $\hbg F\cong \Q[x_1,x_2]$ \eqref{eq:xi}.
\end{lemma}
\begin{proof}
	Every element in the $F$-th component of $M$ \eqref{eq:m1} can be written as:
	\begin{equation*}
		[([\frac {p(x_1,x_2)}{x_1^kx_2^h}], 0, 0, \dots)]
		\in \quotient {(\bigoplus_{i\geq 1}\frac {\eginv \qxy}{\ehinv i \qxy})}{\Delta}
	\end{equation*}
	with $k,h\geq 1$ and only the first component of the representative of the class different from zero.
	
	Every element in $(\nt)_{n_1}$ can be written as $x_1^sx_1^r\otimes f$ for $s\geq 0$, with $t_{n_1}^rf$ regular at $\c\gen {n_1}$ and that does not vanish on $\c\gen {n_1}$.
	
	Therefore every element in $(M\tens 1 \nt)_F$ is sum of elements of the form:
	\begin{equation}
		\label{eq:classelement}
		[([x_1^{-k}x_2^{-h}], 0, 0, \dots)]\tens {\Q[x_1]}(x_1^r\otimes f)= [([x_1^{-k}x_2^{-h}], 0, 0, \dots)]\tens {\Q[x_1]}(x_1^k\otimes (ft_{n_1}^{r-k})).
	\end{equation}
	The equality \eqref{eq:classelement} is true since under the isomorphism:
	\begin{equation}
		\frac {\eginv \qxy}{\ehinv 1\qxy}\tens {\Q[x_1]}(\nt)_{n_1} \cong \ehinv 1\qxy \tens {\Q[x_1]}\quotient {\kt}{\O_{\c\gen {n_1}}} 
	\end{equation}
	the two elements in the left hand side: $[x_1^{-k}x_2^{-h}] \tens {\Q[x_1]}(x_1^r\otimes f)$ and $[x_1^{-k}x_2^{-h}] \tens {\Q[x_1]}(x_1^k\otimes (ft_{n_1}^{r-k}))$ are sent to the same element in the right hand side.
\end{proof}

\appendix
\section{Algebraic models}
\label{appendix:algmodel}
In this section we present a self-contained account of algebraic models for tori of any rank, therefore in all this section $G=\T^r$ is an r-dimensional torus, with $r\geq 0$. We also specify that all the modules over graded rings are graded and all the maps between graded modules are graded maps. As in the rest of the paper everything is rationalized without comment.

Algebraic models are a useful tool to study rational equivariant cohomology theories. The main idea is to define an abelian category $\A(G)$ and homology functor from the category of rational $G$-equivariant orthogonal spectra to $\A(G)$:
\begin{equation}
	\label{eq:homologyfunctor2}
	\pi_*^{\A}:\sp G_{\Q} \to \A(G)
\end{equation}
equipped with an Adams spectral sequence to compute maps in the homotopy category of rational $G$-spectra. More precisely the values of the theory may be calculated by a spectral sequence:
\begin{equation}
	\label{eq:ass2}
	\Ext^{*,*}_{\A}(\pi_*^\A(X),\pi_*^\A(Y)) \Longrightarrow [X,Y]^G_{*}.
\end{equation}
In the case of tori we have a zig-zag of Quillen-equivalences \cite[Theorem 1.1]{john:torusship}):
\begin{equation}
	\label{eq:shipley}
	\sp G_{\Q}\simeq_Q d\A(G),
\end{equation}
where $d\A(G)$ is the model category of differential graded objects in $\A(G)$. Therefore we can build rational $G$-equivariant cohomology theories simply constructing objects in $d\A(G)$.

\subsection{Definition of the rings}
We start by defining the rings needed for the construction of $\A(G)$ \cite[Section 3.A.]{john:torusI}. We write $\F$ for the family of finite subgroups of $G$.

\begin{definition}
	For every connected subgroup $H$ of $G$ define the collection:
	\begin{equation*}
		\quotient {\F}H:=\{\tilde{H}\leq G \mid H  \text{ finite index in } \tilde H \}
	\end{equation*}
	and the ring:
	\begin{equation}
		\label{eq:ofhi}
		\oef H := \prod_{\tilde H \in \F/H}H^*\left(B\left( \quotient G{\tilde H}\right)\right)
	\end{equation}
\end{definition}

\begin{rem}
	Note $\oef G=\Q$ and $\oef 1=\of$.
\end{rem}
Any containment of connected subgroups $K\subseteq H$ induces an inflation map $\oef H \to \oef K$, defined in the following way. 
\begin{definition}
	The inclusion $K\subseteq H$ of connected subgroups defines a quotient map $q:\quotient GK \to \quotient GH$, and hence
	\begin{equation}
		\label{eq:qstar}
		q_*:\quotient \F{K}\to \quotient \F{H}.
	\end{equation}
	For any $\tilde{K}\in \quotient \F{K}$ define the $\tilde{K}$-th component of the inflation map $\oef H \to \oef K$ to be the composition:
	\begin{equation}
		\label{eq:inflationmap}
		\oef H=\prod_{\tilde H \in \F/H}\hbg {\tilde H} \to \hbg {q_*\tilde{K}} \to \hbg {\tilde K}
	\end{equation}
	given by projection onto the term $\hbg {q_*\tilde{K}}$ followed by the inflation map induced by the quotient $\quotient G{\tilde K} \to \quotient G{q_*\tilde K}$.
	
\end{definition}
\begin{rem}
	\label{rem:inflationmap}
	In particular for any connected subgroup $H$ we have an inflation map induced by the inclusion of the trivial subgroup:
	\begin{equation}
		\label{eq:inflation}
		i_H:\oef H \to \of 
	\end{equation}
	which is a split monomorphism of $\oef H$-modules \cite[Proposition 3.1]{john:torusII}. As a consequence $\of$ is an $\oef H$-module for every connected subgroup $H$.
\end{rem}

\subsection{Euler classes}
\label{sub:eulerclass}
Fundamental elements of these rings are Euler classes of representations of $G$, used in the localization process. For any complex representation $V$ of $G$ we want to define its Euler class $e(V)\in \of$ \cite[Section 3.B.]{john:torusI}. We require them to be multiplicative: $e(V\oplus W)=e(V)e(W)$, therefore it's enough to define Euler classes for one dimensional complex representations $V$.
\begin{definition}
	For a one dimensional complex representation $V$ of $G$, define its Euler class $e(V)\in \of$ as follows. For every finite subgroup $F$ the $F$-th component $e(V)_F\in \hbg F$ is:
	\begin{equation}
		\label{eq:eulerclass}
		e(V)_F=
		\begin{cases*}
			1 & if $V^F=0$ \\
			\bar e(V^F) & if $V^F\neq 0$,
		\end{cases*}
	\end{equation}
\end{definition}
where $\bar e(V^F)\in H^2(BG/F)$ is the classical equivariant Euler class for the $G/F$ representation $V^F$.

\begin{definition}
	For any connected subgroup $H$ of $G$ define the multiplicatively closed subset of $\of$:
	\begin{equation}
		\label{eq:eh}
		\E_H:=\{e(V)\mid V^H=0\}.
	\end{equation}
\end{definition}

\subsection{The $2$-torus}
We are mainly interested in the case of the 2-torus, therefore let us compute explicitly rings and Euler classes in this case. Recall that $\{H_i\}_{i\geq 1}$ is the collection of connected closed codimension $1$ subgroups of $\T^2$ and $H_i^j$ is the subgroup with $j$-components and identity component $H_i$. Let $z_i$ be a character of $\T^2$ with kernel $H_i$ and $z_i^j$ be $z_i$ post-composed with the multiplication by $j$ map (it has kernel $H_i^j$). In the case of the 2-torus:
\begin{itemize}
	\item $\oef {\T^2}=\Q $
	\item $\oefh i= \prod_{j\geq 1} H^*(B\T^2/{H_i^j})$
	\item $\of= \prod_{F} H^*(B\T^2/F)$ , where $F$ runs through all the finite subgroups of $\T^2$.
\end{itemize}
For every $i,j\geq 1$:
\begin{equation}
	\label{eq:cij}
	\begin{split}
		H^*(B\T^2/{H_i})&\cong \Q[c_{i}]\\
		H^*(B\T^2/{H_i^j})&\cong \Q[c_{ij}]
	\end{split}
\end{equation}  
where $c_{i}=e(z_i)$ and $c_{ij}=e(z_i^j)$ both of degree $-2$ are the Euler classes of the characters $z_i$ and $z_i^j$ (more precisely of the one dimensional complex representations defined by those characters).

\begin{definition}
	\label{def:ni}
	For every finite subgroup $F$ of $\T^2$ and every index $i\geq 1$ define $n_i=n_i(F)$ to be the only positive integer such that $H_i^{n_i}$ is generated by $H_i$ and $F$: $\gen {F,H_i}=H_i^{n_i}$.
\end{definition}

Every finite subgroup $F$ can be written as the intersection of two codimension one subgroups of $\T^2$. Therefore for every $F$ there exists two different integers $A=A(F)\geq 1$ and $B=B(F)\geq 1$ such that
\begin{equation}
	\label{eq:splitting}
	F=H_A^{n_A}\cap H_B^{n_B}.
\end{equation}
\begin{choice}
	For any finite subgroup $F$ we choose a pair of positive integers $(A,B)$ that give the decomposition \eqref{eq:splitting}.
\end{choice}

By \eqref{eq:splitting} we obtain the decomposition:
\begin{equation}
	\label{eq:xaxb}
	H^*(B\T^2/F)\cong H^*(B\T^2/{\ha})  \otimes H^*(B\T^2/{\hb}) \cong \Q[\xa,\xb].
\end{equation} 
Where $\xa:=e(\za)$, $\xb:=e(\zb)$ have both degree $-2$ are the Euler classes respectively of $\za$ and $\zb$.
\begin{definition}
	\label{def:xi}
	For every $i\geq 1$ define
	\begin{equation}
		\label{eq:xi}
		x_i:=e(z_i^{n_i})\in H^2(B\T^2/F).
	\end{equation} 
	Notice it is an integral linear combination of $\xa$ and $\xb$.
\end{definition}

\begin{rem}
	\label{rem:inflationcomp}
	With these choices of coordinates \eqref{eq:cij} and \eqref{eq:xaxb} the inflation map \eqref{eq:inflation} on the $F$-th component of the target $\of$ can be easily described:
	\begin{equation}
		\label{eq:inflationhi}
		\oefh i = \prod_{j\geq 1}H^*(B\T^2/{H_i^j}) \to H^*(B\T^2/{H_i^{n_i}}) \rightarrowtail H^*(B\T^2/F).
	\end{equation}
	The first map of \eqref{eq:inflationhi} is the projection onto the $n_i$-th component since $q_*(F)=\gen {H_i,F}=H_i^{n_i}$ by definition of the index $n_i$. The second map of \eqref{eq:inflationhi} is the natural inclusion of $\Q$-algebras sending the generator $c_{i,n_i}$ to $x_i$, since by \eqref{eq:xi} they are the same Euler class $e(z_i^{n_i})$ for the two different rings.
\end{rem}

\subsection{Description of $\A(G)$}
We briefly recap the description of $\A(G)$ \cite[Definition 3.9]{john:torusI}. The objects of $\A(G)$ are sheaves of modules over the poset of connected subgroups of $G$ with inclusions.

\begin{definition}
	\label{def:objectag}
	An object $X\in \A(G)$ is specified by the following pieces of data:
	\begin{enumerate}
		\item For every connected subgroup $H$ an $\oef H$-module $\phi^HX$.
		\item For every containment of connected subgroups $K\subseteq H$ an $\oef K$-modules map:
		\begin{equation}
			\label{eq:strmap}
			\phi^KX\to \E_{\quotient HK}^{-1}\oef K\tens {\oef H}\phi^HX.
		\end{equation}
	\end{enumerate}
	Then $X$ is a sheaf over the space of connected subgroups of $G$. This specifically means that for every connected subgroup $H$, the sheaf $X$ has value the $\of$-module:
	\begin{equation}
		\label{eq:valuex}
		X(H):=\E_H^{-1}\of \tens {\oef H} \phi^HX,
	\end{equation}
	and that for every containment $K\subseteq H$ of connected subgroups, $X$ has a structure map of $\of$-modules:
	\begin{equation}
		\label{eq:strmapofx}
		\beta_K^H:X(K) \to X(H).
	\end{equation}
	The map \eqref{eq:strmapofx} is obtained tensoring the $\oef K$-modules map \eqref{eq:strmap} with the $\oef K$-module $\E_K^{-1}\of$. Moreover $X$ satisfies the condition that for every connected subgroup $H$ the $\of$-modules structure map $\beta_1^H:X(1)\to X(H)$ is the map inverting the multiplicatively closed subset of Euler classes $\E_H$ \eqref{eq:eh}.
\end{definition}
\begin{rem}
	Notice that by Remark \ref{rem:inflationmap} the inflation map $i_K$ makes $\of$ an $\oef K$-module. Moreover the structure map \eqref{eq:strmap} is well defined from \eqref{eq:strmap}, since:
	\begin{equation}
		\label{eq:strtensored}
		\E_K^{-1}\of\tens {\oef K} \E_{\quotient HK}^{-1}\oef K \cong \E_H^{-1}\of.
	\end{equation}
\end{rem}

\begin{ex}
	\label{ex:2tor}
	For the 2-torus an object $X\in \A(\T^2)$ has the shape:
	\begin{equation*}
		\begin{tikzcd}
			& X(\T^2) &  \\
			X(H_1) \arrow[ur, "\beta_{H_1}^{\T^2}"] & X(H_2) \arrow[u] & \dots \arrow[ul] \\
			& X(1) \arrow[ur] \arrow[u] \arrow[ul, "\beta_1^{H_1}"] & 
		\end{tikzcd}=
		\left[ \begin{tikzcd}
			X(\T^2) \\
			X(H_i) \arrow[u] \\
			X(1) \arrow[u]  
		\end{tikzcd}\right] 
	\end{equation*}
	with infinitely many values $X(H_i)$ in the middle row, one vertex $X(\T^2)$ and one value $X(1)$ at the bottom level. By \eqref{eq:valuex}: 
	\begin{equation}
		\label{eq:valuet2}
		\begin{split}
			X(\T^2) & =\etinv \tens {\Q} \phi^{\T^2}X \\
			X(H_i) & =\ethinv i \of \tens {\oefh i}\phi^{H_i}X \\
			X(1) & =\of \tens {\of} \phi^1X = \phi^1X
		\end{split}
	\end{equation}
\end{ex}
\begin{nota}
	\label{notation:tensi}
	A tensor product with no ring specified will always mean over $\Q$. For any $i\geq 1$ we denote $\tens i=\tens {\oefh i}$ the tensor product over the ring $\oefh i$ or when we are considering the $F$-th component: $\tens i=\tens {\hbti {n_i}}$.
\end{nota}
\begin{ex}
	\label{ex:eulerclasses}
	For the 2-torus using the coordinates we have defined (\eqref{eq:cij}, \eqref{eq:xaxb}, and  \eqref{eq:xi}), we can easily describe the localizations at the Euler classes:
	\begin{itemize}
		\item In $\eghinv i \hbti j$ we are inverting the Euler class $c_{ij}$.
		\item In $\ehinv i\hbt F$ we are inverting all the Euler classes $x_j$ with $j\geq 1$ and $j\neq i$.
		\item In $\etinv \hbt F$ we are inverting all the Euler classes $x_j$ with $j\geq 1$. 
	\end{itemize}
\end{ex}
\begin{ex}
	There is a structure sheaf $\O\in \A(G)$ \cite[Definition 3.3]{john:torusI} obtained using as modules the base rings: $\phi^H\O=\oef H$, and as structure maps the natural inclusions:
	\[
	\oef K \rightarrow \E_{\quotient HK}^{-1}\oef K\tens {\oef H}\oef H.
	\]
\end{ex}
\begin{definition}
	A morphism $f:X\to Y$ in the category $\A(G)$ is the data of a (graded) $\oef H$-module map $\phi^Hf:\phi^HX \to \phi^HY$ for every connected subgroup $H$, compatible with the structure maps of $X$ and $Y$ (it makes the evident commutative diagrams between different levels commute \cite[Definition 3.6]{john:torusI}).
\end{definition}
\begin{rem}
	\label{rem:bottomlevel}
	A morphism $f:X\to Y$ in $\A(G)$ is almost determined by what it does at the trivial subgroup level $f(1):X(1)\to Y(1)$ (that we will call bottom level). This is because for any connected subgroup $H$ the map $f$ at the $H$-th level $f(H):X(H)\to Y(H)$ is then $\einv H f(1)$. Therefore properties like injectivity, surjectivity or exactness for a sequence of morphisms can be checked at the bottom level.
\end{rem}

\subsection{Injectives in $\A(G)$}
The injective objects in $\A(G)$ that we will use are constant below a certain connected subgroup $H$, and zero elsewhere \cite[Section 4.A.]{john:torusI}.
\begin{definition}
	An object $X\in \A(G)$ is concentrated below a connected subgroup $H$ if $X(K)=0$ for every connected subgroup $K\nsubseteq H$. We denote $\A(G)_H$ the full subcategory of $\A(G)$ of objects concentrated below $H$.
\end{definition}
\begin{definition}
	If $H$ is a connected subgroup of $G$, and $T$ is a graded torsion $\oef H$-module, define $f_H(T)\in \A(G)$ to be the constant sheaf below $H$ with the following values:
	\begin{equation}
		\label{eq:deffh}
		f_H(T)(K):=
		\begin{cases*}
			\einv H\of \tens {\oef H} T & if $K \subseteq H$ \\
			0 & if $K \nsubseteq H$
		\end{cases*}
	\end{equation}
	and structure maps either identities or zero.
\end{definition}
\begin{rem}
	We require $T$ to be torsion so that when we invert $\einv K$ for $K\nsubseteq H$ we obtain zero. Therefore this requirement can be dropped when $H=G$.
\end{rem}
\begin{lemma}[Lemma 4.1 of \cite{john:torusI}]
	\label{lem:adj}
	For any connected subgroup $H$ of $G$ there is an adjunction:
	\begin{equation*}
		\begin{tikzcd}
			\A(G)_H\ar[r,bend left,"\phi^{H}",""{name=A, below}] & 
			\quad \quad\text{Tors-$\oef H$-Mod} \ar[l,bend left,"f_H",""{name=B,above}] \ar[from=A, to=B, symbol=\dashv]
		\end{tikzcd}
	\end{equation*}
	where the left adjoint is the evaluation $\phi^H$. Moreover for any torsion $\oef H$-module $T$, and object $X\in \A(G)$ we have:
	\begin{equation}
		\label{eq:adjunction}
		\Hom_{\oef H}(\phi^HX, T)\cong \Hom_{\A(G)}(X, f_H(T)).
	\end{equation}
\end{lemma}
This in order allows us to transfer torsion injectives $\oef H$-modules into injective objects in $\A(G)$. First notice that if we are given for every $\tilde H\in \quotient {\F}H$ an $\hbg {\tilde H}$-torsion module $T(\tilde H)$, then $\bigoplus_{\quotient {\F}H}T(\tilde H)$ is naturally a torsion $\oef H$-module, with the action given component by component.
\begin{corollary}[Lemma 5.1 of \cite{john:torusI}]
	\label{cor:inj1}
	Suppose $\{H_i\}_{i\geq 1}$ is the collection of all connected subgroups of $G$ of a fixed dimension. If $f_{H_i}(T_i)$ is injective for every $i$, then so is $\bigoplus_{i\geq 1}f_{H_i}(T_i)$.
\end{corollary}
\begin{corollary}[Corollary 5.2 of \cite{john:torusI}]
	\label{cor:inj2}
	If for every $\tilde H \in \quotient {\F}H$, $T(\tilde H)$ is a graded torsion injective $\hbg {\tilde H}$-module. Then $f_H(\bigoplus_{\quotient {\F}H}T(\tilde H))$ is injective in $\A(G)$.
\end{corollary}

\subsection{Spheres of complex representations}
\label{sub:spheres}
We can now define the fundamental homology functor $\pi_*^{\A}$ \cite[Definition 1.4]{john:torusI}. Given a rational $G$-spectrum $X$ we can define the sheaf $\pi_*^{\A}(X)\in \A(G)$ that on a connected subgroup $H$ takes the value
\begin{equation}
	\label{eq:paistarfunctor}
	\pai \A(X)(H):=\einv H\of \tens {\oef H}\pi_*^{G/H}(\defhp H \wedge \fix HX).
\end{equation}
Where $\fix H$ is the geometric fixed point functor, $E\F_+$ is the universal space for the family $\F$ of finite subgroups with a disjoint basepoint added, and $\defp=F(E\F_+, S^0)$ is its functional dual (The function spectrum of maps from $E\F_+$ to $S^0$). 
\begin{ex}
	\label{ex:modelsphere}
	For the sphere spectrum we obtain the structure sheaf $\O$ \cite[Theorem 1.5]{john:torusI}:
	\[
	\pai \A(S^0)(H) =\einv H\of \tens {\oef H}\oef H
	\]
\end{ex}
Given a complex representation $V$ of $G$ we want to make explicit the object $\pai \A(S^V)$ \cite[Section 2.B.]{john:torusII}. To do so we need first to introduce suspensions:
\begin{definition}
	If $V$ is an $n$-dimensional complex representation of $G$, divide the family $\F$ of finite subgroups of $G$ into $n+1$ disjoint sets $\F_i$, where 
	\begin{equation*}
		\F_i:=\{F\in \F \mid \dim_{\C}(V^F)=i\}.
	\end{equation*}
	If $M$ is an $\of$-module, define the $V$-th suspension of $M$ to be the $\of$-module:
	\begin{equation*}
		\susp VM:=\bigoplus_{i=0}^n \susp{2i}e_{\F_i}M,
	\end{equation*}
	where $e_{\F_i}\in \of$ is the idempotent associated to $\F_i$ (it has a one in the $F$-th component if $F\in \F_i$ and zero everywhere else).
\end{definition}
The value of the sheaf $\pai \A(S^V)$ at a connected subgroup $H$ is:
\begin{equation}
	\label{eq:sphere}
	\pai \A(S^V)(H) =\einv H\of \tens {\oef H}\Sigma^{V^H}\oef H.
\end{equation}
To describe the structure maps it is convenient to use the suspension of the units:
\begin{equation*}
	\i {V^H} :=\Sigma^{V^H}(1)\in \Sigma^{V^H}\oef H
\end{equation*}
so that for every inclusion of connected subgroups $K\subseteq H$ the structure map $\beta_K^H$ is determined by the suspended unit:
\begin{equation}
	\label{eq:structuremaps}
	\beta_K^H(\i {V^K})= e(V^K-V^H)^{-1}\otimes \i {V^H}
\end{equation}
where the difference of the two representations simply means the orthogonal complement:
\[
V^K=V^H\oplus (V^K-V^H).
\]
\begin{rem}
	\label{rem:appendixnegative}
	The content of this section applies also in the case of a virtual complex representation $V=V_0-V_1$. The only thing to specify is the Euler class $e(V)=e(V_0)/e(V_1)$. As a result \eqref{eq:sphere} becomes:
	\begin{equation*}
		\pai \A(S^V)(H) =\einv H\of \tens {\oef H}\Sigma^{V_0^H-V_1^H}\oef H.
	\end{equation*}
\end{rem}

\printbibliography

@book{john:mem,
    author = {Greenlees, J.P.C.},
    title = {Rational $S^1$-equivariant stable homotopy theory},
    series = {Mem. American Math. Soc.},
    publisher = {American Mathematical Society},
    year = {1999},
    volume = {661},
}

@article{john:elliptic,
    author = {Greenlees, J.P.C.},
    title = {Rational $S^1$-equivariant elliptic cohomology},
    journal = {Topology},
    year = {2005},
    volume = {44},
    pages = {1213-1279},
}

@article{john:torusI,
    author = {Greenlees, J.P.C.},
    title = {Rational torus-equivariant stable homotopy I: calculating groups of stable maps},
    journal = {JPAA},
    year = {2008},
    volume = {212},
    pages = {72-98},
}

@article{john:torusII,
    author = {Greenlees, J.P.C.},
    title = {Rational torus-equivariant stable homotopy II: the algebra of localization and inflation},
    journal = {JPAA},
    year = {2012},
    volume = {216},
    pages = {2141-2158},
}

@article{john:torusship,
    author = {Greenlees, J.P.C. and Shipley, B.E.},
    title = {An algebraic model for rational torus-equivariant spectra},
    journal = {Journal of Topology},
    year = {2018},
    volume = {11},
    pages = {666-719},
}

@incollection{johnmay:equivariant,
   title =     {Equivariant stable homotopy theory},
   author =    {Greenlees, J.P.C. and May, J.P.},
   publisher = {North-Holland},
   editor    = {Ioan Mackenzie James},
  booktitle   = {The Handbook of Algebraic Topology},
   pages = {277-323},
   year =      {1995},
   edition =   {1},
}

@misc{lurie:ellipticI,
    author={Lurie, Jacob},
    title={Elliptic cohomology I: Spectral Abelian Varieties},
    year={2018},
    howpublished = {\url{http://people.math.harvard.edu/~lurie/papers/Elliptic-I.pdf}},
}

@misc{lurie:ellipticII,
    author={Lurie, Jacob},
    title={Elliptic cohomology II: Orientations},
    year={2018},
    howpublished = {\url{http://www.math.harvard.edu/~lurie/papers/Elliptic-II.pdf}},
}

@misc{lurie:ellipticIII,
    author={Lurie, Jacob},
    title={Elliptic cohomology III: Tempered cohomology},
    year={2019},
    howpublished = {\url{https://www.math.ias.edu/~lurie/papers/Elliptic-III-Tempered.pdf}},
}

@misc{gepner:onequiv,
    author = {Gepner, David and Meier, Lennart},
    title = {On equivariant topological modular forms},
    year = {2020},
    howpublished = {arXiv:2004.10254},
}

@book{hartshorne,
   title =     {Algebraic Geometry},
   author =    {Hartshorne, Robin},
   publisher = {Springer-Verlag New York},
   series =    {Grauduate Texts in Mathematics},
   volume =    {52},
   year =      {1977},
   edition =   {1},
}

@misc{ginzburg:elliptic,
    author = {Ginzburg, V. and Kapranov, M. and Vasserot, E.},
    title = {Elliptic algebras and equivariant elliptic cohomology},
    year = {1995},
    howpublished = {arXiv:q-alg/9505012},
}

@article{groj:delocalised,
    author = {Grojnowski, Ian},
    title = {Delocalised equivariant elliptic cohomology},
    journal = {Elliptic cohomology: Geometry, applications, and higher chromatic analogues},
    year = {2007},
    publisher = {London Mathematical society},
}

@article{stolz,
author = {Stolz, Stephan and Teichner, Peter},
journal = {Proceedings of Symposia in Pure Mathematics},
publisher = {American Mathematical Society},
year = {2011},
pages = {279-340},
title = {Supersymmetric field theories and generalized cohomology},
volume = {83},
}

@article{daniel:supersymmetric,
author = {Berwick-Evans, Daniel},
year = {2021},
month = {09},
pages = {2287-2384},
title = {Supersymmetric field theories and the elliptic index theorem with complex coefficients},
volume = {25},
journal = {Geometry and Topology},
}

@misc{strickland:formal,
    author = {Strickland, N.P.},
    title = {Formal groups},
    year = {2019},
    howpublished = {\url{https://strickland1.org/courses/formalgroups/fg.pdf}},
}

@book{beauville, 
	edition={2}, 
	series={London Mathematical Society Student Texts}, 
	title={Complex Algebraic Surfaces}, 
	publisher={Cambridge University Press}, 
	author={Beauville, Arnaud}, 
	year={1996}, 
}

@book{Liu,
    author = {Liu, Qing},
    title = {Algebraic geometry and arithmetic curves},
    series = {Oxford Graduate Texts in Mathematics},
    volume = {6},
    publisher = {Oxford University Press},
    year = {2002},
}

@article{kervaire,
    author = {Hill, M.A. and Hopkins, M.J. and Ravanel, D.C.},
    title = {On the nonexistence of elements of Kervaire invariant one},
    journal = {Annals of Mathematics},
    year = {2016},
    volume = {184},
    pages = {1-262},
}

@article{dev:delocalised,
    author = {Devoto, J.A.},
    title = {Equivariant elliptic cohomology and finite groups},
    journal = {Michigan Math. J.},
    year = {1996},
    volume = {43},
}

@article{lan:periodic,
    author = {Landweber, P.S.},
    title = {Elliptic cohomology and modular forms},
    journal = {Elliptic curves and modular forms in algebraic topology},
    year = {1988},
    series = {Lecture notes in Mathematics},
    volume = {1326},
    pages = {55-68},
    publisher = {Springer Berlin},
}

@article{atiyah:complet,
author = {M.F. Atiyah and G.B. Segal},
title = {{Equivariant $K$-theory and completion}},
volume = {3},
journal = {Journal of Differential Geometry},
number = {1-2},
publisher = {Lehigh University},
pages = {1 -- 18},
year = {1969},
}

@article{lan:ell,
    author = {Landweber, P.S. and Ravenel, D.G. and Stong, R.E.},
    title = {Periodic cohomology theories defined by elliptic curves},
    journal = {Contemp. Math},
    year = {1995},
    pages = {317--337}
}

@article{quillen:formal,
    author = {Quillen, D.G.},
    title = {On the formal group laws of unoriented and complex cobordism theory},
    journal = {Bull. Amer. Math. Soc.},
    year = {1969},
    volume = {75},
    pages = {1293-1298},
}

@misc{barucco:t2,
  howpublished = {arXiv:2205.09660},
  author = {Barucco, Matteo},
  title = {An algebraic model for rational $T^2$-equivariant elliptic cohomology},
  publisher = {arXiv},
  year = {2022},
}

@misc{vakil:notes,
    author={Vakil, Ravi},
    title={The Rising Sea: Foundations of Algebraic Geometry},
    year={2017},
    howpublished = {\url{http://math.stanford.edu/~vakil/216blog/FOAGnov1817public.pdf}},
}

@misc{barnes:rational,
  author = {Barnes, David},
  title = {Rational Equivariant Spectra},
  year = {2008},
  howpublished = {arXiv:0802.0954},
}

@article{dev:algebraicdescr,
    author = {Devoto, J.A.},
    title = {An algebraic description of the elliptic cohomology of classifying spaces},
    journal = {Journal of Pure and Applied Algebra},
    year = {1998},
    volume = {130},
}

\end{document}